\documentclass{amsart}
\usepackage{hyperref}
\usepackage{graphicx}
\usepackage{amsmath}
\usepackage{amssymb}
\usepackage{amsthm}
\usepackage{dsfont}
\usepackage{mathrsfs}
\usepackage{sansmath}
\usepackage{xcolor}
\usepackage{comment}

\numberwithin{equation}{section}
\newtheorem{theorem}{Theorem}[section]
\newtheorem{lemma}[theorem]{Lemma}
\newtheorem{proposition}[theorem]{Proposition}
\newtheorem{corollary}[theorem]{Corollary}
\newtheorem{definition}[theorem]{Definition}
\newtheorem{assumption}[theorem]{Assumption}

\newtheorem{remark}[theorem]{Remark}
\newtheorem{example}[theorem]{Example}

\DeclareMathOperator*{\tr}{\mathrm{tr}}

\DeclareMathOperator*{\argmin}{arg\,min}

\DeclareMathOperator{\Var}{Var}
\DeclareMathOperator{\rank}{rank}

\newcommand{\indep}{\perp \!\!\! \perp}
\newcommand{\Cpl}{\mathrm{Cpl}}

\newcommand{\R}{\mathbb{R}}
\newcommand{\dist}{\mathsf{dist}}
\newcommand{\AW}{\mathcal{AW}}
\newcommand{\BW}{\mathrm{BW}}
\newcommand{\FP}{\mathrm{FP}}
\newcommand{\cFG}{\mathcal{FG}}

\newcommand{\cFP}{\mathcal{FP}}
\newcommand{\cB}{\mathcal{B}}
\newcommand{\cC}{\mathcal{C}}
\newcommand{\cD}{\mathcal{D}}
\newcommand{\cF}{\mathcal{F}}

\newcommand{\cW}{\mathcal{W}}
\newcommand{\FG}{\mathrm{FG}}

\newcommand{\Cov}{\textup{Cov}}
\newcommand{\cN}{\mathcal{N}}

\newcommand{\cI}{\mathcal{I}}
\newcommand{\cL}{\mathcal{L}}
\newcommand{\cP}{\mathcal{P}}
\newcommand{\cT}{\mathcal{T}}
\newcommand{\cX}{\mathcal{X}}

\newcommand{\dd}{\mathrm{d}}
\newcommand{\AB}{\mathrm{AB}}

\newcommand{\bE}{\mathbb{E}}
\newcommand{\E}{\mathbb{E}}
\newcommand{\bF}{\mathbb{F}}
\newcommand{\bG}{\mathbb{G}}

\newcommand{\diag}{\textup{diag}}
\newcommand{\bP}{\mathbb{P}}
\newcommand{\bR}{\mathbb{R}}
\newcommand{\bX}{\mathbb{X}}
\newcommand{\bY}{\mathbb{Y}}

\newcommand{\norm}[1]{|#1 |}

\newcommand{\normF}[1]{\|#1\|_{\mathrm{F}}}

\newcommand{\normnuc}[1]{\|#1\|_*}

\newcommand{\mgle}{\mathrm{mgle}}
\newcommand{\Markov}{\mathrm{Markov}}
\begin{document}

\title{Adapted optimal transport between filtered Gaussian processes}

\author{Madhu Gunasingam}
\address{Department of Statistical Sciences, University of Toronto}
\email{madhu.gunasingam@mail.utoronto.ca}

\author{Ting-Kam Leonard Wong}
\address{Department of Statistical Sciences, University of Toronto}
\email{tkl.wong@utoronto.ca}


\date{}


\begin{abstract}
We continue the study of adapted optimal transport in the discrete-time Gaussian setting. To this end, we introduce a space of filtered Gaussian processes where both the randomness and the flow of information are driven by a Gaussian white noise. On this space, the adapted $2$-Wasserstein distance ($\AW_2$) admits a variational representation as a constrained orthogonal Procrustes problem between Cholesky factors. Furthermore, the resulting quotient space is the $\AW_2$-completion of the space of Gaussian distributions on the path space. We also characterize explicitly the $\AW_2$-projections onto the subspaces of Gaussian martingales. Next, we analyze the adapted Brenier coupling---a multivariate generalization of the Knothe--Rosenblatt coupling that serves as a myopic solution to the adapted transport problem, and compute its transport cost. Utilizing a Gaussian random matrix framework, we investigate the asymptotic behavior of transport costs as the time horizon grows; notably, we establish that the transport costs of all Gaussian bicausal couplings are asymptotically equivalent, whereas the classical Bures--Wasserstein distance is strictly smaller. Finally, we demonstrate that the adapted analogue of Gelbrich's lower bound fails in general, and we identify a sufficient martingale difference condition under which the bound is recovered.
\end{abstract}

\maketitle

\section{Introduction}
{\it Adapted optimal transport} tailors optimal transport \cite{V03, V08} to the setting of stochastic analysis, by incorporating the direction of time and hence the flow of information. It provides not only a suitable framework for quantifying model uncertainty in mathematical finance and stochastic dynamic programming, but also natural loss functions for modelling sequential data. We refer the reader to \cite{lassalle2018causal} for the general theory of adapted (or (bi)causal) optimal transport, and \cite{BBYZ2017, BBP2024} for its specialization in discrete-time that is the focus of this paper. Necessary concepts, including the adapted $2$-Wasserstein distance $\AW_2$ between filtered processes, will be recalled in Section \ref{sec:AOT}.

In this paper, we continue the study of adapted optimal transport in the Gaussian discrete-time setting, where explicit computations---seldom available in other cases---are possible. Given the importance of Gaussian processes in probability and data science, a solid understanding of the Gaussian setting will likely stimulate further applications of adapted optimal transport. This line of research was initiated by the recent works \cite{GW24, AHP2024} whose authors computed explicitly the adapted $2$-Wasserstein distance between non-degenerate Gaussian distributions on the path space. Here, we extend these results to the setting of filtered processes which allow degeneracy and different choices of the filtration. See \cite{jiang2025transfer} for another extension to continuous-time Gaussian processes.

Let $N, d \geq 1$ be given integers, where $N$ is the number of time steps and $d$ is the spatial dimension. The basic objects of our study are {\it filtered Gaussian processes} (Definition \ref{def:filtered.Gaussian.process}), which are filtered processes of the form
\begin{equation} \label{eqn:filtered.Gaussian.intro}
\bG^{a,L} := (\bR^{Nd}, \mathcal{B}(\bR^{Nd}), \cN_{Nd}(0, I), \bF = (\cF_t)_{t=1}^N, X = a + L\epsilon),
\end{equation}
where $(\bR^{Nd}, \mathcal{B}(\bR^{Nd}), \cN_{Nd}(0, I))$ is the path space equipped with the standard Gaussian distribution, $\epsilon = (\epsilon_t)_{t = 1}^N$ is the canonical process which is a Gaussian white noise, and $\bF$ is the filtration induced by $\epsilon$. On this filtered probability space, we consider a Gaussian process $X = (X_t)_{t = 1}^N$ given as a column vector by
\[
X = a + L\epsilon,
\]
where $a = (a_t)_{t = 1}^N \in \bR^{Nd} \cong (\bR^d)^N$ is the mean and $L = (L_{s,t})_{s, t = 1}^N \in \bR^{Nd \times Nd} \cong (\bR^{d \times d})^{N \times N}$, the {\it Cholesky factor}, has lower triangular $d \times d$ blocks (written $L \in \mathscr{L}(N, d)$).\footnote{In Table \ref{tab:notation} we gather the various spaces of matrices, and their symbols, used in the paper.} The lower triangular structure guarantees that $X$ is adapted to $\bF$, but $\bF$ may be strictly larger than the natural filtration $\bF^X$ induced by $X$. The law of $X$ is Gaussian with mean $a$ and covariance matrix $A = LL^{\intercal}$, denoted by $\cN_{Nd}(a, A)$. Note that a given covariance matrix may be realized by multiple Cholesky factors representing different information structures. Thus, a filtered Gaussian process is a richer object than a Gaussian distribution. The set of all filtered Gaussian processes is denoted by $\cFG$. 

Throughout, we let $\bX = \bG^{a, L}, \bY = \bG^{b, M} \in \cFG$ be filtered Gaussian processes. The corresponding stochastic processes are denoted by $X = a + L\epsilon^X$ and $Y = b + M\epsilon^Y$. Recentering if necessary, we may assume $a = b = 0$. We let $A = LL^{\intercal}$ and $B = MM^{\intercal}$ be the corresponding covariance matrices.

\subsection{Summary of contributions}
We adopt a unified approach to study and compare classical and adapted transports between $\bX$ and $\bY$, with a focus on the following three settings:
\begin{enumerate}
\item[(i)] Classical optimal transport (Section \ref{sec:Wasserstein.Gaussian}) which neglects the filtrations and leads to the $2$-Wasserstein distance $\cW_2$ between distributions on the path space $\bR^{Nd}$. Since the processes are Gaussian, $\cW_2$ reduces to the {\it Bures--Wasserstein distance}, for which a new expression is given in \eqref{eqn:Wasserstein.Cholesky.intro} below.
\item[(ii)] Adapted optimal transport (Section \ref{sec:filtered.Gaussian}) which results in the adapted $2$-Wasserstein distance $\AW_2$. When equipped with $\AW_2$, the set $\cFG$ of filtered processes becomes a pseudo-metric space with rich properties.
\item[(iii)] The {\it adapted Brenier coupling} (Section \ref{sec:adapted.Brenier}) which solves a simplified myopic (step-by-step) version of the adapted optimal transport problem. The adapted Brenier coupling, which can be defined beyond the Gaussian setting, may be regarded as a generalization of the {\it Knothe--Rosenblatt coupling}, and is handy when the adapted Wasserstein distance is difficult to compute explicitly.
\end{enumerate}
In all cases, we provide explicit formulas and characterize the set of optimal Gaussian couplings.\footnote{Non-Gaussian optimal couplings can still be characterized but they are more cumbersome to state. For our purposes Gaussian couplings are sufficient.} Essentially, these transports boil down to different ways of coupling the Gaussian driving noises $\epsilon^X$ and $\epsilon^Y$.

\medskip

To illustrate our approach, we revisit classical optimal transport between Gaussian distributions in Section \ref{sec:Wasserstein.Gaussian}. Even in this setting, our approach leads to an expression of the Bures--Wasserstein distance that, to the best of our knowledge, is new in the literature. Specifically, we show in Proposition \ref{prop:Wasserstein.Cholesky} that
\begin{equation} \label{eqn:Wasserstein.Cholesky.intro}
\cW_2^2( \cN_{Nd}(0, A), \cN_{Nd}(0, B)) = \|L\|_{\mathrm{F}}^2 + \|M\|_{\mathrm{F}}^2 - 2\|L^{\intercal}M\|_*,
\end{equation}
where $\|\cdot\|_{\mathrm{F}}$ and $\|\cdot\|_*$ denote respectively the Frobenius and nuclear norms. 

In Section \ref{sec:filtered.Gaussian} we study various aspects of the space $\cFG$ of filtered processes equipped with the (pseudo-) metric $\AW_2$. Section \ref{sec:sub.filtered.Gaussian} shows that the adapted $2$-Wasserstein distance is given by
\begin{equation} \label{eqn:AW2.intro}
\AW_2^2(\bX, \bY) = \|L\|_{\mathrm{F}}^2 + \|M\|_{\mathrm{F}}^2 - 2\sum_{t = 1}^N \| (L^{\intercal}M)_{t,t} \|_*,
\end{equation}
and has a variational representation in the form of a Procrustes problem, and that $(\cFG, \AW_2)$ is (after taking a suitable quotient) a complete metric space. Independently, these results were also obtained by the authors of \cite{ABGHP26}, who moved on to study further geometric properties of the tangent space and curvature. Here, we focus on other aspects. In Section \ref{sec:minimal.Cholesky}, we show that by choosing a suitably defined {\it minimal Cholesky factor} for each covariance matrix, we obtain the adapted $2$-Wasserstein distance between Gaussian {\it distributions} that are possibly degenerate (Theorem \ref{thm:AW2.Gaussian.distributions}). Furthermore, we show that $\cFG$ (after taking the quotient) can be identified with the $\AW_2$-completion of the set of Gaussian distributions on $\bR^{Nd}$. This gives a theoretical justification of our definition of $\cFG$.

In Section \ref{sec:mgle.Markov}, we study the subspaces of martingales and Markov processes in $\cFG$. We also consider a projection problem with respect to $\AW_2$ (Proposition \ref{prop:right.invariance}), and show that the martingale projection has an explicit solution (Corollary \ref{cor:mgle.proj}).

Section \ref{sec:adapted.Brenier} studies the adapted Brenier coupling. In Theorem \ref{thm:adapted.Brenier}, we show that its transport cost is given by
\begin{equation} \label{eqn:Brenier.cost}
\|L\|_{\mathrm{F}}^2 + \|M\|_{\mathrm{F}}^2 - 2 \sum_{t = 1}^N \tr ( (L^{\intercal}M)_{t,t} P_t),
\end{equation}
where each $P_t$ is a certain correlation matrix that maximizes the trace $\tr(L_{t,t}^{\intercal} M_{t,t}P_t)$. It is interesting to note that this does not define a squared distance. 

In Section \ref{sec:random.matrix}, we compare these transport costs in a probabilistic framework where $L$ and $M$ are random matrices with i.i.d.~Gaussian entries. Considering the regime where $d$ is fixed and $N \rightarrow \infty$, we show that the transport costs of all Gaussian bicausal couplings grow asymptotically at the same rate (Theorem \ref{thm:ensemble.equiv}), while the (squared) Bures--Wasserstein distance is, as expected, strictly smaller (Theorem \ref{thm:strict-gap}). 

The above results are consistent with the general phenomenon that classical and adapted optimal transports can have very different behaviours. In Section \ref{sec:Gelbrich} we provide yet another instance of this phenomenon. Namely, we show that {\it Gelbrich's lower bound} \cite{G1990} of the $2$-Wasserstein distance does not extend to the adapted setting. That is to say, for $\mu, \nu \in \cP_2(\bR^{Nd})$ with mean zero and covariance matrices $A$ and $B$, it is generally {\it not} true that
\begin{equation} \label{eqn:adapted.Gelbrich.intro}
\AW_2(\mu, \nu) \geq \AW_2( \cN_{Nd}(0, A), \cN_{Nd}(0, B)).
\end{equation}
(Gelbrich's classical bound states that \eqref{eqn:adapted.Gelbrich.intro} holds if we replace $\AW_2$ by $\cW_2$ on both sides.) In Theorem \ref{thm:adapted.Gelbrich}, we identify a {\it martingale difference condition} which is restrictive but guarantees that \eqref{eqn:adapted.Gelbrich.intro} holds.

\begin{table}[h!]
\begin{center}
\begin{tabular}{c|l}
Symbol & Meaning  \\ 
\hline
$\mathscr{L}(n)$ & Lower triangular matrices \\
$\mathscr{L}_+(n)$, $\mathscr{L}_{++}(n)$ & Elements of $\mathscr{L}(n)$ with non-negative  \\
&(resp.~positive) diagonal entries  \\
$\mathscr{L}(N, d)$ & Block lower triangular matrices  \\ 
$\mathscr{L}_\mgle(N, d)$ & Elements of $\mathscr{L}(N,d)$ \\
& corresponding to martingales\\
$\mathscr{L}_{\mathrm{Markov}}(N, d)$ & Elements of $\mathscr{L}(N,d)$ \\
& corresponding to Markov processes\\
$\tilde{\mathscr{L}}(N, d)$, $\tilde{\mathscr{L}}_{++}(N, d)$ & Spaces of $\dist_{\AW}$-equivalence classes \\
$\mathscr{S}_+(n)$, $\mathscr{S}_{++}(n)$ & Positive semidefinite (resp.~positive \\
& definite) matrices \\
$\mathscr{C}(n)$ & Correlation matrices \\
$\mathscr{O}(n)$ & Orthogonal matrices  \\
$\mathscr{O}(N, d)$ & Block diagonal matrices with  \\
 & orthogonal diagonal blocks  \\
$\mathscr{P}(C)$ & Optimal correlation matrices for $C$ in\\
& Proposition \ref{prop:trace} \\
$\mathscr{Q}(L, M)$ & Optimal orthogonal matrices for $L$, $M$\\
& in Theorem \ref{thm:Procrustes}
\end{tabular}
\end{center}
\caption{Spaces of matrices used in the paper. We denote by $n$ a generic dimension, $N$ the temporal dimension (number of time steps), and $d$ the spatial dimension.} \label{tab:notation}
\end{table}

\section{Preliminaries} \label{sec:prelim}
\subsection{Adapted optimal transport} \label{sec:AOT}
We begin by recalling some fundamental definitions in the theory of adapted optimal transport \cite{BBP2024}.  In the following, we consider stochastic processes with values in $\bR^d$ and indexed by $t \in [N] := \{1, \ldots, N\}$. A vector $x$ in the path space $\R^{Nd} \cong (\R^d)^N$ decomposes as $x=(x_{{1}},\dots,x_{{N}})$, where $x_t \in\R^d$ represents the $t$-th temporal block. In calculations, we regard $x$ and $x_t$ (and other vectors) as column vectors unless otherwise stated. The Euclidean norm of $x$ is denoted by $\|x\|_2 = (x^{\intercal}x)^{1/2}$, where $^\intercal$ denotes the matrix transpose. We use $0$ and $I$ to denote respectively the zero vector (or matrix) and the identity matrix of suitable dimensions.  Given a matrix $A\in\R^{Nd\times Nd}$ and $s, t \in [N]$, we use $A_{s,t} \in \R^{d \times d}$ to denote the $(s, t)$-th temporal block. When we need to refer to the $(i, j)$-th scalar entry of a matrix $C$ in the usual sense, we use the notation $C_{[i, j]}$. Similarly, $x_{[i]}$ is the $i$-th component of a vector $x$ (here $[i]$ should not be confused with $\{1, \ldots, i\}$). The law (distribution) of a random element $X$ is denoted by $\cL(X)$. Given a metric space $\cX$, we let $\cP(\cX)$ be the set of Borel probability measures on $\cX$, and $\cP_2(\cX)$ be the subset of those that have finite second moment. 


\begin{definition}[Filtered process]
A filtered process is a five-tuple 
\begin{equation} \label{eqn:filtered.process}
\bX = \left(\Omega^{\bX}, \cF^{\bX}, \bP^{\bX}, \bF^{\bX} = (\cF_t^{\bX})_{t = 1}^N, X = (X_t)_{t = 1}^N\right),
\end{equation}
where $(\Omega, \cF, \bP, \bF)$ is a filtered probability space and $X$ is an $\R^d$-valued stochastic process on $(\Omega, \cF, \bP)$ adapted to $\bF$. We let $\cFP_2$ be the family of all filtered processes $\bX$ with $\cL(X) \in \cP_2(\bR^{Nd})$. 
\end{definition}


A (Borel) probability measure $\mu$ on the path space $\bR^{Nd}$ has a canonical representation as a filtered process.

\begin{definition}[Canonical representation of a distribution as a filtered process] \label{def:embedding}
For $\mu \in \cP_2(\bR^{Nd})$, we define 
\begin{equation} \label{eqn:law.as.filtered.process}
\bX^{\mu} := (\bR^{Nd}, \cB(\bR^{Nd}), \mu, (\cF_t)_{t = 1}^N, (X_t)_{t = 1}^N),
\end{equation}
where $X_t(\omega) = \omega_t$ is the canonical process on $(\bR^{Nd}, \cB(\bR^{Nd}))$ and has law $\cL(X) = \mu$, and $\cF_t = \sigma(X_s: s \leq t)$ is the natural filtration induced by $X$. We call $\bX^{\mu} \in \cFP_2$ the canonical representation of $\mu$ as a filtered process.
\end{definition}

Next, we describe how to couple two filtered processes. To take into account the underlying filtrations, we will require the coupling to be {\it bicausal}. Intuitively, this means that the coupling cannot look into the future. Given filtered processes $\bX$, $\bY$ and $0 \leq s, t \leq N$, we let $\cF_{t,s}^{\bX, \bY}$ be the product $\sigma$-algebra on $\Omega^{\bX} \times \Omega^{\bY}$ defined by $\cF_{t,s}^{\bX, \bY} := \cF_t^{\bX} \otimes \cF_s^{\bY}$, where by convention $\cF_0^{\bX} = \{\emptyset, \Omega^{\bX}\}$ and $\cF_0^{\bY} = \{\emptyset, \Omega^{\bY}\}$.

\begin{definition}[Bicausal coupling] \label{def:bicausal.coupling}
Let $\bX$ and $\bY$ be filtered processes. A coupling between $\bX$ and $\bY$ is a probability measure $\pi$ on $(\Omega^{\bX} \times \Omega^{\bY}, \cF^{\bX} \otimes \cF^{\bY})$ whose marginals are $\bP^{\bX}$ and $\bP^{\bY}$. We say that $\pi$ is bicausal if for every $1 \leq t \leq N$ we have, under $\pi$,
\begin{equation} \label{eqn:bicausal}
\begin{split}
&\cF_{N, 0}^{\bX, \bY} \indep \cF_{0,t}^{\bX, \bY} \text{ given } \cF_{t,0}^{\bX, \bY}, \text{ and } \cF_{0, N}^{\bX, \bY} \indep \cF_{t,0}^{\bX, \bY} \text{ given } \cF_{0,t}^{\bX, \bY}.
\end{split}
\end{equation}
We let $\Cpl_{\mathrm{bc}}(\bX, \bY)$ be the set of bicausal couplings between $\bX$ and $\bY$. When $\bX = \bX^{\mu}$ and $\bY = \bX^{\nu}$ are canonical representations of $\mu, \nu \in \cP(\bR^{Nd})$, we write $\Cpl_{\mathrm{bc}}(\mu, \nu) := \Cpl_{\mathrm{bc}}(\bX^{\mu}, \bX^{\nu})$.
\end{definition}

A bicausal coupling is a coupling (in the ordinary sense) of the underlying probability measures $\bP^{\bX}$ and $\bP^{\bY}$; this induces, via the process $(X, Y)$ on the product space $\Omega^{\bX} \times \Omega^{\bY}$, a coupling between the laws $\cL(X)$ and $\cL(Y)$. The set of bicausal couplings is always nonempty as it contains the product coupling $\bP^{\bX} \otimes \bP^{\bY}$. The adapted $2$-Wasserstein distance is defined in terms of the value of the following bicausal optimal transport problem.

\begin{definition}[Adapted $2$-Wasserstein distance] \label{def:AW2}
The adapted $2$-Wasserstein distance $\AW_2(\bX, \bY)$ between filtered processes $\bX, \bY \in \cFP_2$ is defined by
\begin{equation} \label{eqn:AW2}
\AW_2(\bX, \bY) := \inf_{\pi \in \Cpl_{\mathrm{bc}}(\bX, \bY)} \bE_{\pi}\left[ \|X - Y\|_2^2 \right]^{\frac{1}{2}}.
\end{equation}
The adapted $2$-Wasserstein distance between laws $\mu, \nu\in \cP_2(\bR^{Nd})$ is defined by
\begin{equation} \label{eqn:AW2.laws}
\AW_2(\mu, \nu) := \AW_2(\bX^{\mu}, \bX^{\nu}),
\end{equation}
where $\bX^{\mu}, \bX^{\nu} \in \cFP_2$ are the canonical representations defined by \eqref{eqn:law.as.filtered.process}.
\end{definition}


\begin{definition}[Equivalence classes of filtered processes]
We define $\FP_2 := \cFP_2/\sim$, where $\bX \sim \bY$ if and only if $\AW_2(\bX, \bY) = 0$. We often use the same symbol $\bX$ to denote a filtered process in $\cFP_2$ and its equivalence class $[\bX]$ which is an element of $\FP_2$.
\end{definition}

A motivation for considering filtered processes, rather than simply distributions on the path space, is that $(\cP_2(\bR^{Nd}), \AW_2)$ is not a complete metric space \cite{BBEP20}. In \cite[Theorems 1.2--1.3]{BBP2024}, it was shown that the completion of $(\cP_2(\bR^{Nd}), \AW_2)$ can be identified with $(\FP_2, \AW_2)$. Moreover, $(\FP_2, \AW_2)$ is a Polish space that is isometric to the usual $2$-Wasserstein space on a space of nested (conditional) distributions. Nevertheless, the canonical representation, which amounts to using the filtration induced by the process, is a natural and reasonable choice in many applications. For example, in mathematical finance and stochastic optimal control, it makes sense to use the filtration induced by the observable state process, rather than that induced by
the driving noise. 

We end this subsection by noting that although adapted optimal transport problems are difficult to solve in general, an abstract Brenier theorem was recently developed in \cite{BPS25} (also see \cite{PS25}). Also, the bicausal optimal transport problem \eqref{eqn:AW2} can be extended to several filtered processes. The resulting {\it multicausal} optimal transport problem was introduced recently in \cite{acciaio2025multicausal}.

\subsection{Bures--Wasserstein distance via Cholesky factors} \label{sec:Wasserstein.Gaussian} 
In this subsection we review the standard $2$-Wasserstein transport between Gaussian distributions, with a new twist that can be extended to the adapted setting. We let $\mathscr{S}_+(n)$ (resp.~$\mathscr{S}_{++}(n)$) be the set of $n \times n$ positive semidefinite (resp.~strictly positive definite) matrices. For $(a, A) \in \bR^n \times \mathscr{S}_+(n)$, we let $\mathcal{N}_n(a, A) \in \cP_2(\bR^n)$ be the Gaussian distribution on $\bR^n$ with mean $a$ and covariance matrix $A$. We let
\[
\mathscr{N}(n) := \{\cN_n(a, A): a \in \bR^{n}, A \in \mathscr{S}_+(n)\}
\]
be the set of Gaussian distributions on $\bR^n$. In the adapted setting, we write $\mathscr{N}(N, d) \cong \mathscr{N}(Nd)$ to make the time and space dimensions explicit.

The {\it $2$-Wasserstein distance} between $\mu, \nu \in \cP_2(\bR^{n})$ is defined by
\begin{equation} \label{eqn:2.Wasserstein}
\cW_2(\mu, \nu) := \inf_{\pi \in \Cpl(\mu, \nu)} \bE_{\pi}\left[ \|X - Y\|_2^2 \right]^{\frac{1}{2}},
\end{equation}
where, as before, $\bE_{\pi}$ is the expectation under which $(X, Y) \sim \pi$. We say that $\pi \in \Cpl(\mu, \nu)$ is a {\it Brenier coupling} of $(\mu, \nu)$ if it is optimal for \eqref{eqn:2.Wasserstein}. It is well known that if $\mu = \cN_n(a, A), \nu = \cN_n(b, B) \in \mathscr{N}(n)$, then
\begin{equation} \label{eqn:W2}
\cW_2^2(\mu, \nu) = \|a - b\|_2^2 + \dist_{\BW}^2(A, B),
\end{equation}
where $\dist_{\BW}$ is the {\it Bures--Wasserstein distance} on $\mathscr{S}_+(n)$ defined by
\begin{equation} \label{eqn:Bures.Wasserstein}
\dist_{\BW}^2(A, B) := \tr A + \tr B - 2 \tr \left( (A^{\frac{1}{2}}BA^{\frac{1}{2}})^{\frac{1}{2}}  \right).
\end{equation}
Here, $\tr$ denotes the trace and $A^{1/2}$ is the square root of $A$ in $\mathscr{S}_+(n)$. We refer the reader to \cite{BJL19, WMP18, T11} for various analytic and geometric properties of the Bures--Wasserstein distance.

\begin{remark}[Centering] \label{rmk:centering}
In classical and adapted $2$-Wasserstein transport, the difference in means is accounted simply by a translation. For example, we have
\[
\cW_2^2( \cN_n(a, A), \cN_n(b, B)) = \|a - b\|_2^2 + \cW_2^2(\cN_n(0, A), \cN_n(0, B)).
\]
Although some results will be stated for general means and covariance matrices, in  most proofs we will assume $a = b = 0$ to simplify the notation.
\end{remark}

As demonstrated in \cite{GW24, AHP2024}, when dealing with Gaussian distributions it is the {\it Cholesky factor}, rather than the covariance matrix itself, that is essential in adapted optimal transport. To compare the classical and adapted settings, we will express the Bures--Wasserstein distance in terms of the Cholesky factors. 

Before doing so, we recall some concepts and results in linear algebra that will be used throughout the paper. Let $\mathscr{L}(n)$ be the set of $n \times n$ lower triangular matrices, and $\mathscr{L}_+(n)$ (resp.~$\mathscr{L}_{++}(n)$) be the subset of matrices in $\mathscr{L}(n)$ whose diagonal entries are nonnegative (resp.~positive). For any $A \in \mathscr{S}_+(n)$, there exists $L \in \mathscr{L}_+(n)$, called a {\it Cholesky factor} of $A$, such that $LL^{\intercal} = A$. When $A \in \mathscr{S}_{++}(n)$, the Cholesky factor is unique and is an element of $\mathscr{L}_{++}(n)$.

Let a matrix $C \in \bR^{m \times n}$ be given.
\begin{itemize}
    \item The {\it Frobenius} norm is defined by $\|C\|_{\mathrm{F}} := \sqrt{\tr (CC^{\intercal})}$.
    \item The {\it nuclear norm} is defined by $\| C \|_{*} := \tr \big( ( C^{\intercal} C)^{1/2} \big)$.
    \item The {\it spectral norm} is defined by $\|C\|_{2 \rightarrow 2} := \sup_{u \in \bR^n: \|u\|_2 \leq 1} \|Cu\|_2$.
\end{itemize} 
These norms can be expressed in terms of the {\it singular value decomposition} $C = U \Sigma V^{\intercal}$, where $U \in \bR^{m \times m}$ and $V \in \bR^{n \times n}$ are orthogonal (written $U \in \mathscr{O}(m), V \in \mathscr{O}(n)$), and $\Sigma \in \bR^{m\times n}$ is a diagonal matrix whose diagonal entries are the singular values of $C$. We let $\sigma_1(C) \geq \cdots \geq \sigma_k(C) \geq 0$, where $k = \min\{m, n\}$, be the {\it singular values} of $C$ arranged in descending order. Multiplying $U$ and $V$ by permutation matrices (which are orthogonal) if necessary, we may, and will, assume without loss of generality that $\Sigma_{[i, i]} = \sigma_i(C)$. 

Since the trace of
\[
C^{\intercal} C = (U\Sigma V^{\intercal})^{\intercal} (U\Sigma V^{\intercal}) = V \Sigma^2 V^{\intercal}
\]
is equal to that of $\Sigma^2$, we have
\[
\|C\|_{\mathrm{F}} = \left(\sigma_1^2(C) + \cdots + \sigma_k^2(C)\right)^{\frac{1}{2}}.
\]
On the other hand, since $(C^{\intercal} C)^{1/2} = V (\Sigma^\intercal\Sigma)^{1/2} V^{\intercal}$, 
\begin{equation} \label{eqn:nuclear.norm.trace}
\|C\|_{*} = \sigma_{1}(C) + \cdots + \sigma_{k}(C)
\end{equation}
is the sum of the singular values. Finally, the spectral norm is given by (see \cite[Example 5.6.6]{HornJohnson2013}) the largest singular value:
\[
\|C\|_{2 \rightarrow 2} = \sigma_1(C).
\]
It follows that
\begin{equation} \label{eqn:matrix.norm.inequality}
\|C\|_{2\rightarrow 2} \leq \|C\|_{\mathrm{F}} \leq \|C\|_{*} \leq \sqrt{k} \|C\|_{\mathrm{F}}.
\end{equation}
Note that $\|C\| = \|C^{\intercal}\|$ for each of these norms. We also note that for $C \in \bR^{n \times n}$, $\|C\|_{2 \rightarrow 2} \leq 1$ if and only if
\begin{equation} \label{eqn:correlation.matrix}
\begin{bmatrix}
I_n & C \\ C^{\intercal} & I_n
\end{bmatrix} \in \mathscr{S}_+(2n).
\end{equation}
In this case, we call $C$ a {\it correlation matrix}. We let
\[
\mathscr{C}(n) := \{C \in \bR^{n \times n} : \|C\|_{2 \rightarrow 2} \leq 1 \}
\]
be the set of all $n \times n$ correlation matrices.

We also recall {\it von Neumann's trace inequality} (see, for example, \cite[Theorem 7.4.1.1]{HornJohnson2013}). For real matrices, it states that if $C, D \in \mathbb{R}^{n \times n}$, then
\begin{equation} \label{eqn:trace.inequality}
|\tr(CD)| \leq \sum_{i = 1}^n \sigma_i(C) \sigma_i(D).
\end{equation}

The following matrix-analytic result provides a variational interpretation of the nuclear norm. Namely, it generalizes the elementary fact that $|c| = \max_{\rho \in [-1, 1]} c\rho$, $c \in \bR$. The trace in \eqref{eqn:trace.maximization} arises naturally when computing expectations of the form $\bE[\|X - Y\|_2^2]$. Here, we cover the degenerate case and characterize the set of all optimizers. Since this result is seldom stated in its most general form, for completeness we provide a proof in the appendix.

\begin{proposition}[Trace maximization] \label{prop:trace}
Let $C \in \bR^{n \times n}$ and consider the optimization problem
\begin{equation} \label{eqn:trace.maximization}
\max\{ \tr(CP) : P \in \bR^{n \times n}, \|P\|_{2 \rightarrow 2} \leq 1 \}.
\end{equation}
\begin{itemize}
\item[(i)] The optimal value is $\|C\|_*$.
\item[(ii)] Let $C$ have singular value decomposition $U \Sigma V^{\top}$ with the convention that $\Sigma = \diag(\sigma_1(C), \ldots, \sigma_n(C))$. Let $r = \rank(C)$ and write
\[
U = \begin{bmatrix} U_1 & U_0 \end{bmatrix}, \quad V = \begin{bmatrix} V_1 & V_0 \end{bmatrix},
\]
where $U_1, V_1 \in \bR^{n \times r}$ and $U_0, V_0 \in \bR^{n \times (n-r)}$ (if $r = n$ then $U_0$ and $V_0$ are empty). Then the set of optimizers is
\begin{equation} \label{eqn:trace.optimizer}
\mathscr{P}(C) := \{ P = V_1 U_1^{\intercal} + V_0 K U_0^{\intercal} : K \in \bR^{(n - r) \times (n - r)}, \|K\|_{2 \rightarrow 2} \leq 1\}.
\end{equation}
\end{itemize}
\end{proposition}

\begin{remark} \label{rmk:trace}
The set $\mathscr{P}(C)$ always contains orthogonal matrices: choosing $K \in \mathscr{O}(n - r)$ yields $P \in \mathscr{O}(n)$. In particular, letting $K = I_{n - r}$ shows that $P = VU^{\intercal} \in \mathscr{P}(C)$. When $C$ is invertible (so that $r = n$), this is the only element of $\mathscr{P}(C)$.
\end{remark}

We are now ready to express the Bures--Wasserstein distance in terms of Cholesky factors, which
are not even required to be lower triangular. To the best of our knowledge, the representation \eqref{eqn:Bures.Wasserstein.Cholesky} has not appeared in the literature.

\begin{proposition}[Bures--Wasserstein in terms of Cholesky factors] \label{prop:Wasserstein.Cholesky}
Let $A, B \in \mathscr{S}_+(n)$. Write $A = LL^{\intercal}$ and $B = MM^{\intercal}$, where $L, M \in \mathbb{R}^{n\times n}$.\footnote{Given $A$ and $B$, the factors $L, M \in \mathbb{R}^{n \times n}$ always exist but are generally not unique. The stated results hold for each choice of $L$ and $M$.}
\begin{itemize}
\item[(i)] We have
\begin{equation} \label{eqn:Bures.Wasserstein.Cholesky}
\begin{split}
\dist_{\BW}^2(A, B) &=  \|L\|_{\mathrm{F}}^2 + \|M\|_{\mathrm{F}}^2 - 2 \| L^{\intercal} M \|_*.
\end{split}
\end{equation}
\item[(ii)] Let $L^{\intercal}M$ have singular value decomposition $U \Sigma V^{\intercal}$. Consider a coupling $\tilde{\pi}$ of $(\cN_n(0, A), \cN_n(0, B))$ of the form $\tilde{\pi} = \cL(X, Y)$, where $(X, Y) = (L\epsilon^X, M \epsilon^Y)$ and 
\begin{equation*} 
\pi := \cL(\epsilon^X, \epsilon^Y) \in \Cpl(\cN_n(0, I), \cN_n(0, I)).
\end{equation*}
Then $\tilde{\pi}$ is a Brenier coupling between $\cN_n(0, A)$ and $\cN_n(0, B)$ if and only if
\begin{equation} \label{eqn:Wasserstein.coupling}
P := \bE_{\pi}[\epsilon^Y (\epsilon^X)^{\intercal}] \in \mathscr{P}(L^{\intercal}M).
\end{equation}
In particular, we may let $\pi$ be the Gaussian coupling
\[
\begin{bmatrix} \epsilon^X \\ \epsilon^Y \end{bmatrix} \sim \cN_{2n} \left( \begin{bmatrix} 0 \\ 0 \end{bmatrix}, \begin{bmatrix} I & P \\ P^{\intercal} & I \end{bmatrix} \right), \quad P \in \mathscr{P}( L^{\intercal}M).
\]
\end{itemize}
\end{proposition}
\begin{proof}
We prove (i) and (ii) together. From \eqref{eqn:nuclear.norm.trace}, we have $\|L^{\intercal} M\|_* = \tr \Sigma$. We first show that
\[
\tr \Sigma = \tr \left( (A^{\frac{1}{2}}BA^{\frac{1}{2}})^{\frac{1}{2}}  \right) = \tr \left( (B^{\frac{1}{2}}AB^{\frac{1}{2}})^{\frac{1}{2}}  \right),
\]
where the last equality follows from the symmetry of the Bures--Wasserstein distance. Since
\[
(L^{\intercal}M) (L^{\intercal}M)^{\intercal} = L^{\intercal} B L = U \Sigma^2 U^{\intercal},
\]
we have
\[
\tr \Sigma = \tr \left( (L^{\intercal} B L)^{\frac{1}{2}} \right).
\]
Let $H = B^{1/2} L$. Then $H^{\intercal} H = L^{\intercal} B L$ and $HH^{\intercal} = B^{1/2} A B^{1/2}$. Since $H^{\intercal} H$ and $HH^{\intercal}$ are positive semidefinite and have the same eigenvalues $\lambda_i$, taking square root and the trace shows that
\[
\tr \Sigma = \tr \left( (L^{\intercal} B L)^{\frac{1}{2}} \right) = \sum_{i = 1}^n \sqrt{\lambda_i} = \tr \left( (B^{\frac{1}{2}}AB^{\frac{1}{2}})^{\frac{1}{2}}  \right).
\]

If $(X, Y)$ is given by (ii), then
\begin{equation} \label{eqn:XminusY}
\begin{split}
\bE_{\pi} [\|X - Y\|_2^2] = \bE_{\pi} [\|L\epsilon^X - M\epsilon^Y\|_2^2] = \|L\|_{\mathrm{F}}^2 + \|M\|_{\mathrm{F}}^2 - 2 \tr (L^{\intercal}M P).
\end{split}
\end{equation}
By Proposition \ref{prop:trace}, this is equal to $\dist_{\mathrm{BW}}^2(A, B)$ if and only if $P \in \mathscr{P}( L^{\intercal}M)$.
\end{proof}

The Bures--Wasserstein distance can be expressed in terms of an {\it orthogonal Procrustes problem} \cite[Section 7.4.5]{HornJohnson2013}. Our version is a slight variant of \cite[Theorem 1]{BJL19} which is stated in terms of the square roots $A^{1/2}$ and $B^{1/2}$ rather than the Cholesky factors. Probabilistically, our Procrustes problem is equivalent to minimization over deterministic couplings of the form $\epsilon^Y = Q\epsilon^X$, where $\epsilon^X \sim \cN_n(0, I)$ and $Q \in \mathscr{O}(n)$.

\begin{corollary}[Procrustes representation] \label{cor:Wasserstein.Procrustes}
Under the setting of Proposition \ref{prop:Wasserstein.Cholesky}, we have
\begin{equation} \label{eqn:Wasserstein.Procrustes}
\dist_{\BW}(A, B) = \min_{Q \in \mathscr{O}(n)} \|L - M Q\|_{\mathrm{F}},
\end{equation}
and the minimum is attained if and only if $Q \in \mathscr{P}(L^{\intercal}M) \cap \mathscr{O}(n)$.
\end{corollary}
\begin{proof}
For $Q \in \mathscr{O}(n)$, we have
\begin{equation*}
\|L - MQ\|_{\mathrm{F}}^2 = \|L\|_{\mathrm{F}}^2 + \|M\|_{\mathrm{F}}^2 - 2 \tr (L^{\intercal}MQ).
\end{equation*}
The rest follows from Proposition \ref{prop:trace}.
\end{proof}

\section{Filtered Gaussian processes} \label{sec:filtered.Gaussian}
In Section \ref{sec:sub.filtered.Gaussian}, we introduce a natural space $\cFG \subset \cFP_2$ of filtered processes $\bX$ such that each stochastic process $X$ is a Gaussian process. We show that the adapted $2$-Wasserstein distance on $\cFG$ reduces to a pseudo-metric between the means and the Cholesky factors. 
In Section \ref{sec:minimal.Cholesky} we specialize to Gaussian distributions, and show that $(\FG, \AW_2)$, where $\FG := (\cFG/\sim) \subset \FP_2$ is the space of equivalence classes, is the $\AW_2$-completion of Gaussian distributions on $\bR^{Nd}$. 

\subsection{Filtered Gaussian processes} \label{sec:sub.filtered.Gaussian}

Let $\mathscr{L}(N, d)$ be the space of $Nd \times Nd$ matrices whose $d \times d$ blocks are lower triangular (we simply say block lower triangular). That is, if $L = (L_{t,s})_{t,s \in [N]} \in \mathscr{L}(N, d)$ where $L_{t,s} \in \bR^{d \times d}$ is the $(t, s)$-th block, then $L_{t,s} = 0_{d \times d}$ whenever $s > t$. 

\begin{definition}[Filtered Gaussian process] \label{def:filtered.Gaussian.process}
Given a mean vector $a \in \bR^{Nd}$ and a Cholesky factor $L \in \mathscr{L}(N,d)$, we define the filtered process $\bG^{a,L} \in \cFP_2$, called a filtered Gaussian process, by
\begin{equation} \label{eqn:filtered.Gaussian}
\bG^{a,L} := (\bR^{Nd}, \mathcal{B}(\bR^{Nd}), \cN_{Nd}(0, I), \bF = (\cF_t)_{t=1}^N, X = a + L\epsilon),
\end{equation}
where $\epsilon = (\epsilon_t)_{t = 1}^N$ is the canonical process on $(\bR^{Nd}, \mathcal{B}(\bR^{Nd}))$ and $\bF = \bF^\epsilon$ is the canonical filtration induced by $\epsilon$. When $a = 0$, we simply write $\bG^{L} := \bG^{0, L}$.

We let $\cFG := \{ \bG^{a,L}: a \in \bR^{Nd}, L \in \mathscr{L}(N,d)\} \subset \cFP_2$ be the space of all filtered Gaussian processes, and let $\FG \subset \FP_2$ be the space of equivalence classes. We write $\cFG(N, d)$ and $\FG(N, d)$ if we need to emphasize the dimensions.
\end{definition}

In Definition \ref{def:filtered.Gaussian.process}, the underlying filtered probability space is the same for all $a$ and $L$. The randomness and the flow of information are driven by an i.i.d.~sequence $(\epsilon_t)_{t \in [N]}$ of $d$-dimensional standard Gaussian random vectors. Since $L$ is block lower triangular, for each time $t \in [N]$,
\begin{equation} \label{eqn:L.epsilon}
X_t = a_t + L_{t,1} \epsilon_1 + \cdots + L_{t,t} \epsilon_t
\end{equation}
is the mean plus a linear combination of the noises up to and including time $t$.\footnote{This is analogous to continuous-time set-ups for adapted optimal transport between stochastic differential equations driven by Brownian motions; see \cite{CL24} and the references therein.} Note that we allow $L_{t,t}$ to be an arbitrary square matrix since $L$ is only required to be block lower triangular (also see Corollary \ref{cor:lower.triangular}). Clearly, $X_t$ is adapted to $\bF$, the filtration induced by the noise process $\epsilon$. The natural filtration induced by $X$ is equal to $\bF$  if and only if each $L_{t,t}$ is invertible. In particular, if $L = I$ is the identity matrix, then $X = \epsilon$ is a standard Gaussian white noise. In general, the distribution of $X = a + L\epsilon$ is
\begin{equation} \label{eqn:Gaussian.law}
\cL(X) = \cN_{Nd}(a, A), \quad A = LL^{\intercal}.
\end{equation}
Since the decomposition $A = LL^{\intercal}$ is not necessarily unique ($L$ is only block lower-triangular), the same distribution can be realized by different choices of $L$. The following example, taken from the proof of \cite[Proposition 5.2]{GW24}, illustrates this point.

\begin{example} \label{eg:simple.example}
Let $N = 2$ and $d = 1$. For $\theta \in \bR$, consider the filtered Gaussian process $\bX = \bG^{ L(\theta)}$, where
\begin{equation} \label{eqn:L.theta}
L(\theta) = \begin{bmatrix} 0 & 0 \\ \cos\theta & \sin\theta\end{bmatrix}.
\end{equation}
Since
\begin{equation} \label{eqn:A}
L(\theta)L(\theta)^{\intercal} = \begin{bmatrix} 0 & 0 \\ 0 & 1 \end{bmatrix} =: A, 
\end{equation}
the distribution of
\[
X = \begin{bmatrix} X_1 \\ X_2 \end{bmatrix} = \begin{bmatrix} 0 \\ (\cos \theta) \epsilon_1 + (\sin \theta) \epsilon_2 \end{bmatrix} = L(\theta) \epsilon
\]
is $\cN_2(0, A)$. 
The conditional distribution of $X_2$ depends on the available information. Since $X_1 = 0$, we have
\[
X_2 \mid X_1 \sim \cN_1(0, 1).
\]
On the other hand, since $\cF_1 = \sigma(\epsilon_1)$ contains information of $\epsilon_1$, we have
\[
X_2 \mid \cF_1 \sim \cN_1( (\cos \theta) \epsilon_1, \sin^2 \theta),
\]
which clearly depends on $\theta$. 
\end{example}


Thanks to the common filtered probability space in \eqref{eqn:filtered.Gaussian}, it is easy to describe bicausal couplings between elements of $\cFG$. Let $\bX = \bG^{a, L}$ and $\bY = \bG^{b, M}$ be two filtered Gaussian processes. The corresponding stochastic processes are denoted by $X = a + L\epsilon^X$ and $Y = b + M \epsilon^Y$, where $\epsilon^X$ and $\epsilon^Y$ are the Gaussian noise processes. This terminology will be used throughout the paper. Specializing Definition \ref{def:bicausal.coupling} to this context, a probability measure $\pi$ on $\bR^{Nd} \times \bR^{Nd} \cong \bR^{2Nd}$ is a bicausal coupling between $\bX$ and $\bY$ if and only if it satisfies the following properties:\footnote{See \cite[Proposition 5.1]{BBYZ2017}.}
\begin{itemize}
\item[(i)] Both marginals of $\pi$ are $\cN_{Nd}(0, I)$.
\item[(ii)] If $(\epsilon^X, \epsilon^Y) \sim \pi$, then for each $t$, the conditional joint distribution
\[
\bP_{\pi}(\epsilon_{t+1}^X \in \cdot, \epsilon_{t+1}^Y \in \cdot \mid \epsilon_{1:t}^X, \epsilon_{1:t}^Y)
\]
is a coupling of $(\cN_d(0, I), \cN_d(0, I))$.
\end{itemize}
Note that the means and Cholesky factors do not appear in these conditions: $\pi$ couples the underlying driving noises; this induces a coupling (in the ordinary sense) between the processes $X$ and $Y$. In fact, it will be seen that $\AW_2(\bG^{a, L}, \bG^{b, M})$ is attained by some bicausal $\pi$ which is jointly Gaussian. Hence, the following characterization of bicausal Gaussian couplings, proved in \cite{AHP2024}, is handy:

\begin{lemma}[Theorem 2.2 of \cite{AHP2024}] \label{lem:coupling.Gaussian}
Let $\bG^{a, L}, \bG^{b, M} \in \cFG$ and $\pi \in \Cpl_{\mathrm{bc}}(\bG^{a, L}, \bG^{b, M})$. The following are equivalent:
\begin{enumerate}
\item[(i)] $\pi \in \cP(\bR^{Nd} \times \bR^{Nd}) \cong  \cP(\bR^{2Nd})$ is jointly Gaussian. (Here and below, we identify $(x, y) \in \bR^{Nd} \times \R^{Nd}$ with the column vector $\begin{bmatrix} x^{\intercal} & y^{\intercal} \end{bmatrix}^{\intercal} \in \bR^{2Nd}$.)
\item[(ii)] There exists a block diagonal matrix $P =\textup{diag}(P_1,\dots,P_N)$, where each diagonal block $P_t \in \bR^{d \times d}$ satisfies $\|P_t\|_{2\rightarrow 2}\le 1$ (that is, each $P_t$ is a contraction and hence is a correlation matrix), such that
\begin{equation} \label{eqn:joint.Gaussian}
\pi = \pi^P := \cN_{2Nd} \left( \begin{bmatrix} 0 \\ 0 \end{bmatrix}, \begin{bmatrix} I & P \\ P^{\intercal} & I \end{bmatrix} \right).
\end{equation}
\end{enumerate}
\end{lemma}

If $(\epsilon^X, \epsilon^Y) \sim \pi^P$, then  $(\epsilon_1^X, \epsilon_1^Y), \ldots, (\epsilon_N^X, \epsilon_N^Y)$ are jointly independent and, for each $t \in [N]$, we have
\begin{equation} \label{eqn:pi.P.conditional}
\begin{bmatrix}
\epsilon_t^X \\ \epsilon_t^Y
\end{bmatrix}
\sim \cN_{2d} \left( \begin{bmatrix} 0 \\ 0 \end{bmatrix} , \begin{bmatrix} I & P_t \\
 P_t^{\intercal} & I \end{bmatrix} \right).
\end{equation}
A straightforward computation shows that the transport cost under the Gaussian bicausal coupling $\pi^P$, $P = \diag(P_1, \ldots, P_N)$, is given by
\begin{equation} \label{eqn:pi.P.cost}
\begin{split}
\bE_{\pi} [ \| X - Y \|_2^2 ] &= \bE_{\pi} [ \| (a + L\epsilon^X) - (b + M\epsilon^Y) \|_2^2] \\
&= \|a - b\|_2^2 + \|L\|_{\mathrm{F}}^2 + \|M\|_{\mathrm{F}}^2 - 2 \tr ( (L^{\intercal}M) P ) \\
&= \|a - b\|_2^2 + \|L\|_{\mathrm{F}}^2 + \|M\|_{\mathrm{F}}^2 - 2\sum_{t = 1}^N \tr ( (L^{\intercal}M)_{t,t} P_t). 
\end{split}
\end{equation}

\begin{example}[Independent coupling] \label{eg:independent couping}
The independent coupling between $\bX = \bG^{a,L}$ and $\bY = \bG^{b,M}$ is $\pi^O$, where $P$ in  \eqref{eqn:joint.Gaussian} is the zero matrix. As the name suggests, under this coupling the noises $\epsilon^X$ and $\epsilon^Y$ are independent.  From \eqref{eqn:pi.P.cost}, its transport cost is given by 
\begin{equation} \label{eqn:independent.cost}
\begin{split}
\bE_{\pi^O}[\|X - Y\|_2^2] &= \|a - b\|_2^2 + \|L\|_{\mathrm{F}}^2 + \|M\|_{\mathrm{F}}^2.
\end{split}
\end{equation}
\end{example}

\begin{example}[Synchronous coupling] \label{eg:synchronous.coupling}
The synchronous coupling is $\pi^I$, where $P$ is the identity matrix. Equivalently, this is achieved by equating the noise processes: $\epsilon^X = \epsilon^Y = \epsilon \sim \cN_{Nd}(0, I)$. Its transport cost is given by 
\begin{equation} \label{eqn:synchronous.cost}
\begin{split}
\bE_{\pi^I}[\|X - Y\|_2^2] &= \|a - b\|_2^2 + \|L - M\| _{\mathrm{F}}^2 \\
&= \|a - b\|_2^2 + \|L\|_{\mathrm{F}}^2 + \|M\|_{\mathrm{F}}^2 - 2 \tr(L^{\intercal}M).
\end{split}
\end{equation}
This gives the upper bound
\begin{equation} \label{eqn:Frobenius.bound}
\AW_2^2(\bG^{a,L}, \bG^{b,M}) \leq \|a - b\|_2^2 + \|L - M\|_{\mathrm{F}}^2.
\end{equation}
Note that when $\tr(L^{\intercal}M) < 0$ the synchronous coupling is worse than the independent coupling. It is helpful to think of the synchronous coupling as a benchmark bicausal coupling. In Section \ref{sec:adapted.Brenier}, we study the adapted Brenier coupling which may be regarded as an alternative benchmark. 
\end{example}

The following result provides an explicit formula for the adapted $2$-Wasserstein distance between two filtered Gaussian processes and a characterization of optimal bicausal couplings that are jointly Gaussian.\footnote{When some $(L^{\intercal}M)_{t,t}$ is singular for some $t$, there exist non-Gaussian optimal bicausal couplings. These couplings can be described in the manner of \cite[Corollary 4.4]{GW24}. For our purposes, it suffices to restrict to Gaussian couplings which lead to  cleaner statements throughout the paper.} The corresponding Procrustes problem is given in Theorem \ref{thm:Procrustes} below. 

\begin{theorem}[$\AW_2$ between filtered Gaussian processes] \label{thm:AW.filtered.Gaussians}
Let $a, b \in \bR^{Nd}$ and $L, M \in \mathscr{L}(N,d)$.
\begin{enumerate}
\item[(i)] We have
\begin{equation} \label{eqn:AW.filtered.Gaussians}
\begin{split}
\AW_2^2(\bG^{a, L}, \bG^{b,M}) = \|a - b\|_2^2 + \dist_{\mathrm{AW}}^2(L, M),
\end{split}
\end{equation}
where $\dist_{\mathrm{AW}}$ is the pseudo-metric on $\mathscr{L}(N, d)$ defined by
\begin{equation} \label{eqn:dist.AW}
\dist_{\mathrm{AW}}(L, M) := \Big(\|L\|_F^2 + \|M\|_F^2 - 2 \sum_{t=1}^N \|(L^{\intercal}M)_{t,t} \|_{*}\Big)^{1/2}.
\end{equation}
We let $\tilde{\mathscr{L}}(N, d)$ be the quotient space $\mathscr{L}(N, d) / \sim$, where $L \sim M$ if and only if $\dist_{\mathrm{AW}}(L, M) = 0$. We use the same symbol $\dist_{\mathrm{AW}}$ to denote the induced metric.
\item[(ii)]
There exists a Gaussian bicausal coupling that is optimal for $\AW_2(\bG^{a, L}, \bG^{b,M})$. A Gaussian bicausal coupling of the form $\pi^P$ defined by \eqref{eqn:joint.Gaussian}, where $P = \diag(P_1, \ldots, P_N)$ is block diagonal, is optimal if and only if
\[
P_t \in \mathscr{P}( (L^{\intercal} M)_{t,t}), \quad t \in [N].
\]
In particular, we may pick $P_t = V_t U_t^{\intercal}$, so that $P$ is block diagonal with orthogonal diagonal blocks (written $P \in \mathscr{O}(N, d)$).
\end{enumerate}
\end{theorem}
\begin{proof}
The proof is an adaptation of the proofs of \cite[Theorems 2.4--2.5]{AHP2024} (and that of \cite[Theorem 1.1]{GW24}) to the filtered setting. Since the ideas are essentially the same, we only highlight the main steps and omit the details. 

Let $\bX = \bG^{a, L}$ and $\bY = \bG^{b, M}$. First, we use a dynamic programming principle (see \cite[Theorem 3.2]{acciaio2025multicausal} which extends \cite[Proposition 5.2]{BBYZ2017} to the filtered and multicausal setting) to express $\AW_2^2(\bX, \bY)$ as the value of an iterated optimization problem.

Define, for elements $\omega^{\bX} = \omega_{1:N}^{\bX}$ and $\omega^{\bY} = \omega^{\bY}_{1:N}$ of $(\bR^{d})^N$ which denote possible realizations of $\epsilon^X$ and $\epsilon^Y$,
\[
V_N(\omega^{\bX}_{1:N}, \omega^{\bY}_{1:N}) := \|(a + L \omega_{1:N}^{\bX}) - (b + M\omega_{1:N}^{\bY})\|_2^2,
\]
and define, inductively backward in time, 
\[
V_t(\omega_{1:t}^{\bX}, \omega_{1:t}^{\bY}) := \inf_{\pi_{t+1}} \int_{\bR^d \times \bR^d} V_{t+1}(\omega^{\bX}_{1:(t+1)}, \omega^{\bY}_{1:(t+1)}) \dd \pi_{t+1}(\omega_{t+1}^{\bX}, \omega_{t+1}^{\bY}),
\]
where $\omega_{1:t}^{\bX}, \omega_{1:t}^{\bY} \in (\bR^d)^t$ and the infimum is over
\begin{equation} \label{eqn:Gaussian.marginal}
\pi_{t+1} \in \Cpl(\cN_{d}(0, I), \cN_{d}(0, I)).
\end{equation}
Note that in $\bX$, $\cN_{d}(0, I)$ is the conditional distribution of the noise $\epsilon_{t+1}^X$ at time $t+1$ given $\epsilon_{1:t}^X$ (same for $\bY$). Then, $V_0$ (note that $\omega_{1:0}^{\bX}$ and $\omega_{1:0}^{\bY}$ are empty) is equal to $\AW_2^2(\bX, \bY)$.

Next, we observe that since the terminal value $V_N$ is a quadratic function of $\omega^{\bX}$ and $\omega^{\bY}$, and the marginals in \eqref{eqn:Gaussian.marginal} are always standard Gaussian, it can be shown by an induction, backward in time, that for each $t$ there exists an optimal coupling $\pi_{t+1}$ which is jointly Gaussian. 

Thus, to find $\AW_2(\bX, \bY)$, it suffices to optimize over Gaussian bicausal couplings. By Lemma \ref{lem:coupling.Gaussian}, any Gaussian $\pi \in \Cpl_{\mathrm{bc}}(\bX , \bY)$ has the form $\pi = \pi^P$ given by \eqref{eqn:joint.Gaussian}, where $P = \diag(P_1, \ldots, P_N)$ is block diagonal with $\|P_t\|_{2 \rightarrow 2} \leq 1$ for each $t$. From \eqref{eqn:pi.P.cost}, its transport is given by
\[
\|a - b\|_2^2 + \|L\|_{\mathrm{F}}^2 + \|M\|_{\mathrm{F}}^2 - 2\sum_{t = 1}^N \tr ( (L^{\intercal} M)_{t,t} P_t).
\]
It remains to solve
\[
\max_{P_t \in \bR^{d \times d} : \|P_t\|_{2 \rightarrow 2} \leq 1} \tr ( (L^{\intercal} M)_{t,t} P_t)
\]
individually for each $t \in [N]$. The rest follows from Proposition \ref{prop:trace}.
\end{proof}

\begin{example}
For $\theta, \phi \in \bR$, define $L(\theta), L(\phi) \in \mathscr{L}(2, 1)$ as in  \eqref{eqn:L.theta}. By Theorem \ref{thm:AW.filtered.Gaussians}, we have
\[
\AW_2^2(\bG^{L(\theta)}, \bG^{L(\phi)}) = 2(1 - |\cos \theta \cos \phi| - |\sin \theta \sin \phi|). 
\]
\end{example}

By Theorem \ref{thm:AW.filtered.Gaussians}, $(\FG, \AW_2)$ is isometric to the product space $\bR^{Nd} \times \tilde{\mathscr{L}}(N, d)$ equipped with the metric
\begin{equation} \label{eqn:metric.product.space}
(\|a - b\|_2^2 + \dist_{\mathrm{AW}}^2([L], [M]))^{1/2}.
\end{equation}
Observe that $\dist_{\AW}$ acts independently on the column blocks of $L$ and $M$. Let $L_{\cdot, t} \in \bR^{Nd \times d}$ be the $t$-th column block of $L$. We define $M_{\cdot, t}$ similarly. Rearranging \eqref{eqn:dist.AW}, we have
\begin{equation} \label{eqn:dist.AW.alternative}
\dist_{\mathrm{AW}}^2(L, M) = \sum_{t = 1}^N \left( \|L_{\cdot, t}\|_{\mathrm{F}}^2 + \|M_{\cdot, t}\|_{\mathrm{F}}^2 - 2 \| (L_{\cdot,t})^{\intercal} M_{\cdot, t} \|_*\right).
\end{equation}

The degeneracy of $\dist_{\mathrm{AW}}$, which defines the equivalence relation $L \sim M$, can be explicitly characterized in terms of multiplication by block orthogonal matrices. 

\begin{proposition}[Degeneracy of $\dist_{\mathrm{AW}}$] \label{prop:dist.AW.degeneracy}
For $L, M \in \mathscr{L}(N, d)$ the following are equivalent:
\begin{itemize}
\item[(i)] $\dist_{\mathrm{AW}}(L, M) = 0$, that is, $L \sim M$.
\item[(ii)] There exists $Q = \diag(Q_1, \ldots, Q_N)$, with $Q_t \in \mathscr{O}(d)$ for all $t \in [N]$, such that
\begin{equation} \label{eqn:equivalence.class}
M = LQ.
\end{equation}
That is, $M_{s,t} = L_{s,t}Q_t$ (or $L_{s,t} = M_{s,t} Q_t^{\intercal}$) for $s, t \in [N]$. We denote by $\mathscr{O}(N, d) \cong \prod_{t = 1}^N \mathscr{O}(d)$  the set of $Nd \times Nd$ block diagonal matrices whose diagonal blocks are orthogonal.
\end{itemize}
In particular, we have 
\begin{equation} \label{eqn:zero.same.cov}
\dist_{\AW}(L, M) = 0 \Rightarrow LL^{\intercal} = MM^{\intercal}
\end{equation}
(that is, the covariance matrices coincide) and
\begin{equation} \label{eqn:dist.invariance}
\dist_{\AW}(L, M) = \dist_{\AW}(L, MQ)
\end{equation}
for $L, M \in \mathscr{L}(N, d)$ and $Q \in \mathscr{O}(N, d)$.
\end{proposition}
\begin{proof}
(ii) $\Rightarrow$ (i): Suppose that \eqref{eqn:equivalence.class} holds. Then, $M_{s, t} = L_{s,t} Q_t$ for all $t,s \in [N]$. By direct computation, for each $t \in [N]$ we have $(L^{\intercal}M)_{t, t} = (L^{\intercal}L)_{t,t} Q_t$. Since the nuclear norm is invariant under (left or right) multiplication by an orthogonal matrix, we have $\|(L^{\intercal}M)_{t, t}\|_{*} = \| (L^{\intercal}L)_{t,t} \|_*$. Also, we have
\[
\|M\|_{\mathrm{F}}^2 = \tr( M^{\intercal} M) = \tr(MM^{\intercal}) = \tr( LQQ^{\intercal}L^{\intercal}) = \tr(LL^{\intercal}) = \|L\|_{\mathrm{F}}^2.
\]
It follows that
\[
\dist_{\mathrm{AW}}(L, M) = \dist_{\mathrm{AW}}(L, L) = 0.
\]
(i) $\Rightarrow$ (ii): We first derive a lower bound of $\dist_{\AW}(L, M)$ for $L, M \in \mathscr{L}(N, d)$. By the trace inequality and the Cauchy--Schwarz inequality, we have
\begin{equation} \label{eqn:trace.and.CS}
\begin{split}
\|(L_{\cdot,t})^{\intercal} M_{\cdot, t}\|_* &= \sum_{i = 1}^d \sigma_i ((L_{\cdot,t})^{\intercal} M_{\cdot, t}) \\
&\leq \sum_{i = 1}^d \sigma_i(L_{\cdot, t}) \sigma_i(M_{\cdot, t}) \\
&\leq \left( \sum_{i = 1}^d \sigma_i^2(L_{\cdot, t}) \right)^{\frac{1}{2}} \left( \sum_{i = 1}^d \sigma_i^2(M_{\cdot, t}) \right)^{\frac{1}{2}} \\
&= \| L_{\cdot, t} \|_{\mathrm{F}} \| M_{\cdot, t}\|_{\mathrm{F}}.
\end{split}
\end{equation}
Putting this into \eqref{eqn:dist.AW.alternative}, we get
\begin{equation} \label{eqn:AW.lower.bound}
\begin{split}
\dist_{\mathrm{AW}}^2(L, M) &\geq \sum_{t = 1}^N (\| L_{\cdot, t}\|_{\mathrm{F}}^2 + \|M_{\cdot, t}\|_{\mathrm{F}}^2 - 2    \| L_{\cdot, t} \|_{\mathrm{F}} \| M_{\cdot, t}\|_{\mathrm{F}}) \\
&= \sum_{t = 1}^N (\| L_{\cdot, t} \|_{\mathrm{F}} - \|M_{\cdot, t}\|_{\mathrm{F}})^2.
\end{split}
\end{equation}

Suppose that $\dist_{\mathrm{AW}}(L, M) = 0$. Then, all inequalities in \eqref{eqn:trace.and.CS} and \eqref{eqn:AW.lower.bound} become equalities. From \eqref{eqn:AW.lower.bound}, we have $\|L_{\cdot, t}\|_{\mathrm{F}} = \|M_{\cdot, t}\|_{\mathrm{F}}$.

On the other hand, for $t \in [N]$, let 
\begin{equation} \label{eqn:polar.decomposition}
(L_{\cdot, t})^{\intercal} M_{\cdot, t} = Q_t P_t,
\end{equation}
where $Q_t \in \mathscr{O}(d)$ and $P_t = \big(((L_{\cdot, t})^{\intercal} M_{\cdot, t})^{\intercal} (L_{\cdot, t})^{\intercal} M_{\cdot, t})\big)^{1/2} \in \mathscr{S}_+(d)$, be the polar decomposition of $(L^{\intercal} M)_{t, t}$. From \eqref{eqn:polar.decomposition}, $P_t = (L_{\cdot, t}Q_t)^{\intercal} M_{\cdot, t}$. Thus, 
\begin{equation} \label{eqn:Frobenius.inner.product}
\begin{split}
\| (L_{\cdot, t})^{\intercal} M_{\cdot, t} \|_* &= \tr \Big( ( (Q_tP_t)^{\intercal} (Q_tP_t) )^{1/2} \Big) \\
&= \tr P_t = \tr ((L_{\cdot, t} Q_t)^{\intercal} M_{\cdot, t}) \\
&=: \langle L_{\cdot, t} Q_t, M_{\cdot, t} \rangle_{\mathrm{F}},
\end{split}
\end{equation}
which is the Frobenius inner product between $L_{\cdot, t}Q_t$ and $M_{\cdot, t}$.

Combining \eqref{eqn:trace.and.CS} and \eqref{eqn:Frobenius.inner.product}, we have
\begin{equation} \label{eqn:CS}
\langle L_{\cdot, t} Q_t, M_{\cdot, t} \rangle_{\mathrm{F}} = \| L_{\cdot, t} \|_{\mathrm{F}} \| M_{\cdot, t}\|_{\mathrm{F}} = \| L_{\cdot, t} Q_t \|_{\mathrm{F}} \|M_{\cdot, t} \|_{\mathrm{F}},
\end{equation}
where the last equality holds since $Q_t \in \mathscr{O}(d)$. From the equality in the Cauchy--Schwarz inequality in \eqref{eqn:CS}, $L_{\cdot, t} Q_t$ and $M_{\cdot, t}$ are proportional. But since $\|L_{\cdot, t}Q_t \|_{\mathrm{F}} = \|L_{\cdot, t}\|_{\mathrm{F}} = \|M_{\cdot, t}\|_{\mathrm{F}}$, the constant of proportionality is $1$. Thus, we have $M_{\cdot, t} = L_{\cdot, t} Q_{ t}$. It follows that
\[
M = LQ, \quad \quad \text{where } Q := \diag(Q_1, \ldots, Q_N) \in \mathscr{O}(N, d).
\]

Finally, we note that \eqref{eqn:zero.same.cov} is immediate from \eqref{eqn:equivalence.class}, and \eqref{eqn:dist.invariance} is a corollary of the triangle inequality:
\[
|\dist_{\mathrm{AW}}(L, M) - \dist_{\mathrm{AW}}(L, MQ)| \leq  \dist_{\mathrm{AW}}(M, MQ) = 0.
\]
\end{proof}

Proposition \ref{prop:dist.AW.degeneracy} can be rephrased probabilistically as follows. Consider the filtered Gaussian process $\bG^{L}$ in which $X$ is given by $L\epsilon$, $\epsilon \sim \cN_{Nd}(0, I)$. Let $Q = \diag(Q_1, \ldots, Q_N) \in \mathscr{O}(N, d)$, and define $M = LQ$. Observe that
\[
\tilde{\epsilon} := Q^{\intercal} \epsilon = (Q_1^{\intercal} \epsilon_1, \ldots, Q_N^{\intercal} \epsilon_N) \sim \cN_{Nd}(0, I),
\]
and $M \tilde{\epsilon} = LQQ^{\intercal} \epsilon = L\epsilon$. The law of $(\epsilon, \tilde{\epsilon})$ is an element of $\Cpl_{\mathrm{bc}}(\bG^{L}, \bG^{LQ})$ with zero transport cost.\footnote{If $Q \in \mathscr{O}(Nd)$, then $Q\epsilon$ is still standard normal, but the coupling $(\epsilon, Q\epsilon)$ is not bicausal unless $Q$ is also block diagonal; see Lemma \ref{lem:coupling.Gaussian}.} Hence $\AW_2(\bG^{L}, \bG^{LQ}) = 0$. Moreover, this is the only possibility to have $\AW_2(\bG^L, \bG^M) = 0$.

\begin{example}
Let $N = 2$ and $d = 1$. Consider
\[
L = \begin{bmatrix} 0 & 0 \\ 1 & 1 \end{bmatrix} \quad \text{and} \quad M = 
\underbrace{\begin{bmatrix} 0 & 0 \\ 1 & 1 \end{bmatrix}}_\text{$L$} \underbrace{
\begin{bmatrix} -1 & 0 \\ 0 & 1 \end{bmatrix}}_\text{$Q$} =
\begin{bmatrix} 0 & 0 \\ -1 & 1 \end{bmatrix}.
\]
A straightforward computation shows that $\dist_{\mathrm{AW}}(L, M) = 0$ as predicted by Proposition \ref{prop:dist.AW.degeneracy}.
\end{example}

In Definition \ref{def:filtered.Gaussian.process} we allow Cholesky factors that are block lower triangular. In fact, for each equivalence class we can always pick a representative which is lower triangular with non-negative diagonal entries. Thus, from the perspective of metric geometry, it is equivalent to define $\cFG$ using either $\mathscr{L}(N, d)$ or $\mathscr{L}_+(Nd)$.

\begin{corollary} \label{cor:lower.triangular}
For $L \in \mathscr{L}(N, d)$, there exists $\tilde{L} \in \mathscr{L}_+(Nd)$ such that $L \sim \tilde{L}$.
\end{corollary}
\begin{proof}
By the QR decomposition (see \cite[Theorem 2.1.14]{HornJohnson2013}), for each $t \in [N]$ there exist $Q_t \in \mathscr{O}(d)$ and $\tilde{L}_{t, t} \in \mathscr{L}_+(d)$ such that
\[
L_{t,t}^{\intercal} = Q_t \tilde{L}_{t, t}^{\intercal}, \quad \text{or equivalently} \quad \tilde{L}_{t,t} = L_{t,t}Q_t.
\]
Let $Q = \diag(Q_1, \ldots, Q_N) \in \mathscr{O}(N, d)$, and define $\tilde{L} = LQ$. Note that its $(t,t)$-th block is equal to $\tilde{L}_{t,t}$, defined above, which is lower triangular. Thus, $\tilde{L} \in \mathscr{L}_+(Nd)$. By Proposition \ref{prop:dist.AW.degeneracy}, we have $L \sim \tilde{L}$ and the proof is complete.
\end{proof}



\begin{theorem}[Procrustes representation] \label{thm:Procrustes}
For $L, M \in \mathscr{L}(N, d)$, we have
\begin{equation} \label{eqn:Procrustes.representation}
\begin{split}
\dist_{\mathrm{AW}}^2(L, M) = \min_{Q \in \mathscr{O}(N, d)} \|L - MQ\|_{\mathrm{F}}^2 = \sum_{t = 1}^N \min_{Q_t \in \mathscr{O}(d)}  \| L_{\cdot, t} - M_{\cdot, t} Q_t\|_{\mathrm{F}}^2.
\end{split}
\end{equation}
If $(L^{\intercal} M)_{t, t} = (L_{\cdot, t})^{\intercal} M_{\cdot, t}$ has singular value decomposition $U_t \Sigma_t V_t^{\intercal}$, the minimum is attained if and only if $Q_t \in \mathscr{P}( (L^{\intercal}M)_{t,t}) \cap \mathscr{O}(d)$. We let $\mathscr{Q}(L, M)$ be the set of $Q \in \mathscr{O}(N, d)$ that has this form.
\end{theorem}
\begin{proof}
Let $L, M \in \mathscr{L}(N, d)$. For $Q = \diag(Q_1, \ldots, Q_N) \in \mathscr{O}(N, d)$, we have
\begin{equation*}
\begin{split}
\| L - MQ \|_{\mathrm{F}}^2 &= \sum_{t = 1}^N \| L_{\cdot, t} - M_{\cdot, t} Q_t \|_{\mathrm{F}}^2 \\
&= \sum_{t = 1}^N \left( \|L_{\cdot, t}\|_{\mathrm{F}}^2 + \|M_{\cdot, t}\|_{\mathrm{F}}^2 - 2 \tr ( ((L_{\cdot, t})^{\intercal} M_{\cdot, t}) Q_t) \right).
\end{split}
\end{equation*}
We may now apply Proposition \ref{prop:trace} for each $t \in [N]$. 
\end{proof}

\begin{theorem}[Completeness] \label{thm:completeness}
The metric space $(\tilde{\mathscr{L}}(N, d), \dist_{\mathrm{AW}})$ is complete. Equivalently, $\FG$ is a closed subspace of $(\FP_2, \AW_2)$. Moreover,
\[
\tilde{\mathscr{L}}_{++}(N d) := \{ [L] \in \tilde{\mathscr{L}}(N, d) : L \in \mathscr{L}_{++}(Nd)\}
\]
is dense in $\tilde{\mathscr{L}}(N, d)$.
\end{theorem}
\begin{proof}
Let $([L^{(n)}])_{n \geq 1}$ be a Cauchy sequence in $(\tilde{\mathscr{L}}(N, d), \dist_{\mathrm{AW}})$. By construction, $(L^{(n)})_{n \geq 1}$ is $\dist_{\AW}$-Cauchy in $\mathscr{L}(N, d)$. From \eqref{eqn:AW.lower.bound}, we have
\[
\sum_{t = 1}^N (\|L_{\cdot, t}^{(n)}\|_{\mathrm{F}} - \|L_{\cdot, t}^{(m)} \|_{\mathrm{F}} )^2 \leq \dist_{\mathrm{AW}}^2(L^{(n)}, L^{(m)}) \rightarrow 0, \quad n, m \rightarrow \infty.
\]
It follows that $L^{(n)}$ is bounded in the Frobenius norm. By the Bolzano--Weierstrass theorem, there exist a subsequence $L^{(n')}$ and a matrix $L \in \mathscr{L}(N, d)$ (which is closed in Frobenius norm), such that
\[
\lim_{n' \rightarrow \infty} \|L^{(n')} - L\|_{\mathrm{F}} = 0.
\]
Since $\dist_{\mathrm{AW}}$ is bounded above by the Frobenius distance by \eqref{eqn:Frobenius.bound}, we have
\[
\lim_{n' \rightarrow \infty} \dist_{\mathrm{AW}}(L^{(n')}, L) = 0.
\]
Since $([L^{(n)}])_{n \geq 1}$ is $\dist_{\mathrm{AW}}$-Cauchy in $\tilde{\mathscr{L}}(N, d)$, we have $\lim_{n \rightarrow \infty} [L^{(n)}] = [L]$. Hence $\tilde{\mathscr{L}}(N, d)$ is complete. That this is equivalent to the closedness of $\FG$ in $\FP_2$ follows from the isometry 
\[
(a, [L]) \in \bR^{dN} \times \tilde{\mathscr{L}}(N, d) \mapsto [\bG^{a, L}] \in \FG,
\]
where the former space is equipped with the metric defined by \eqref{eqn:metric.product.space}.

To show that $\mathscr{L}_{++}(N d) / \sim$ is dense in $\tilde{\mathscr{L}}(N, d)$, it suffices to show that for any $L \in \mathscr{L}(N, d)$, there exists a sequence $(L^{(n)})_{n \geq 1}$ in $\mathscr{L}_{++}(N, d)$ such that $\lim_{n \rightarrow \infty} \dist_{\mathrm{AW}}(L^{(n)}, L) = 0$. By Corollary \ref{cor:lower.triangular}, we may assume $L \in \mathscr{L}_+(Nd)$. Consider $L^{(n)} = L + I/n \in \mathscr{L}_{++}(Nd)$. By \eqref{eqn:Frobenius.bound} again, we have
\[
\dist_{\mathrm{AW}}(L^{(n)}, L) \leq \|L^{(n)} - L\|_{\mathrm{F}} = \frac{1}{n} \|I\|_{\mathrm{F}} = \frac{\sqrt{Nd}}{n} \rightarrow 0.
\]
\end{proof}

We close this subsection by noting that any two elements of $\FG$ can be joined by {\it some} $\FP_2$-geodesic contained in $\FG$. This property may be called ``weak geodesic convexity'', as opposed to the standard notion of geodesic convexity ({\it all} geodesics are contained in the set). As shown in \cite[Remark 2.6]{ABGHP26}, $\FG$ is {\it not} a geodesically convex subset of $\FP_2$. That is, for certain pairs of elements of $\FG$, there exists a geodesic not contained in $\FG$. 

\begin{proposition}[Weak geodesic convexity]\label{prop:geodesic.convexity}
Let $\bX_0 = \bG^{a_0, L_0}, \bX_1 = \bG^{a_1, L_1} \in \cFG$, where $(a_0, L_0), (a_1, L_1) \in \bR^{Nd} \times \mathscr{L}(N, d)$. Given $Q \in \mathscr{Q}(L_0, L_1)$, define $\hat{L}_1 = L_1 Q \in \mathscr{L}(N, d)$. For $u \in [0, 1]$, let $\hat{\bX}_u = \bG^{\hat{a}_u, \hat{L}_u} \in \cFG$, where
\begin{equation} \label{eqn:geodesic}
\hat{a}_u = (1 - u) a_0 + u a_1 \quad \text{and} \quad \hat{L}_u = L_0 + u (\hat{L}_1 - L_0).
\end{equation}
Then $\AW_2(\bX_0, \hat{\bX}_0) = \AW_2(\bX_1, \hat{\bX}_1) = 0$ and
\begin{equation} \label{eqn:geodesic2}
\AW_2(\hat{\bX}_u, \hat{\bX}_v) = |u - v| \AW_2(\hat{\bX}_0, \hat{\bX}_1), \quad u, v \in [0, 1].
\end{equation}
In particular, $([\hat{\bX}_u])_{u \in [0, 1]}$ is a constant speed geodesic from $[\bX_0]$ to $[\bX_1]$ contained in $\FG$.
\end{proposition}
\begin{proof}
From Remark \ref{rmk:centering}, we may assume $a_0 = a_1 = 0$. Clearly $\hat{\bX}_0 = \bX_0$. Since $Q \in \mathscr{O}(N, d)$, from Proposition \ref{prop:dist.AW.degeneracy} we have $\AW_2(\hat{\bX}_1, \bX_1) = 0$. By Theorem \ref{thm:Procrustes}, $Q$ is optimal for \eqref{eqn:Procrustes.representation} (with $L = L_0$ and $M = L_1$). It follows that
\[
\AW_2(\bX_0, \bX_1) = \dist_{\mathrm{AW}}(L_0, L_1) = \|L_0 - L_1 Q\|_{\mathrm{F}} = \|L_0 - \hat{L}_1\|_{\mathrm{F}}.
\]
Let $u \leq v$. By \eqref{eqn:Frobenius.bound} and \eqref{eqn:geodesic} , we have
\begin{equation*}
\begin{split}
\AW_2(\hat \bX_u, \hat \bX_v) &= \dist_{\mathrm{AW}}(\hat L_u, \hat L_v) \leq \|\hat L_u -  \hat L_v\|_{\mathrm{F}} \\
&= (v - u) \|L_0 - \hat{L}_1\|_{\mathrm{F}} = (v - u) \AW_2(\bX_0, \bX_1). 
\end{split}
\end{equation*}
By the triangle inequality for $\AW_2$, we have
\begin{equation*}
\begin{split}
\AW_2(\bX_0, \bX_1) &\leq \AW_2(\bX_0, \hat \bX_u) + \AW_2(\hat \bX_u, \hat \bX_v) + \AW_2(\hat \bX_v, \bX_1) \\
  &= (u + (v - u) + (1 - v)) \AW_2(\bX_0, \bX_1) = \AW_2(\bX_0, \bX_1).
\end{split}
\end{equation*}
Thus, the above inequalities are all equalities, and the proof is complete.
\end{proof}


\subsection{Minimal Cholesky factor} \label{sec:minimal.Cholesky}
In this subsection, we specialize Theorem \ref{thm:AW.filtered.Gaussians} to compute the adapted $2$-Wasserstein distance between arbitrary Gaussian distributions on $\bR^{Nd}$, or equivalently their canonical representations as filtered processes (Definition \ref{def:embedding}). This relaxes the non-degeneracy assumption in \cite[Theorem 1.1]{GW24} and \cite[Theorem 1.1]{AHP2024}. The key idea is to select, for each covariance matrix, a {\it minimal} Cholesky factor which is consistent with the natural  filtration induced by the corresponding Gaussian process.

\begin{proposition}[Minimal Cholesky factor] \label{prop:Chronological.Cholesky}
Let $n \geq 1$ and $A \in \mathscr{S}_+(n)$. Then there exists a unique $L \in \mathscr{L}_+(n)$ with the following properties:
\begin{enumerate}
\item[(i)] $LL^{\intercal} = A$.
\item[(ii)] If $L_{[i, i]} = 0$ then $L_{[j,i]} = 0$ for all $j \geq i$. That is, the $i$-th column $L_{[\cdot, i]}$ vanishes whenever the diagonal element $L_{[i,i]}$ does.
\end{enumerate}
We call $L$ the minimal Cholesky factor of $A$ and write $L = \mathcal{C}_{\min}(A)$.
\end{proposition}

This result is a slight refinement of the usual Cholesky decomposition, and the usual proof (e.g.~by induction on $n$) goes through with suitable modification to take the zero-column condition (ii) into account. When $A \in \mathscr{S}_{++}(n)$, $\cC_{\min}(A)$ is the unique $L \in \mathscr{L}_{++}(n)$ such that $LL^{\intercal} = A$. When $A$ is singular, there still exists $L \in \mathscr{L}_+(n)$ such that $LL^{\intercal} = A$, but this $L$ is not necessarily unique. The proposition states that uniqueness is restored if we additionally impose the zero-column condition. 

\begin{example} \label{eg:minimal.Choleky}
Consider the matrices $A \in \mathscr{S}_+(2)$ and $L(\theta) \in \mathscr{L}_+(2)$ in Example \ref{eg:simple.example}. It is easy to verify that
\[
\{ L \in \mathscr{L}_+(2): LL^{\intercal} = A\} = \left\{ L(\theta): \theta \in [0, \frac{\pi}{2}] \right\}.
\]
The minimal Cholesky factor of $A$ is given by
\[
\cC_{\min}(A) = L\Big(\frac{\pi}{2}\Big) = \begin{bmatrix} 0 & 0 \\ 0 & 1 \end{bmatrix}.
\]
\end{example}

The minimal Cholesky factor $L = \cC_{\min}(A)$ of $A$ may be singular. In several arguments it is necessary to define a canonical inverse-like matrix of $L$ that we denote by $L^{\ominus}$. Intuitively, if $X = L\epsilon$, then $L^{\ominus}X$ recovers the noises that are active in the product $L\epsilon$. We call $L^{\ominus}$ the {\it chronological inverse} of $L$. 

Here is the precise definition. Let $A \in \mathscr{S}_+(n)$ and $L = \cC_{\min}(A)$. Let
\begin{equation} \label{eqn:cI}
\mathcal{I} = \mathcal{I}(L) := \{ i \in [n]: L_{[i,i]} > 0\},
\end{equation}
and recall that the $i$-th column $L_{[\cdot, i]}$ vanishes whenever $i \notin \mathcal{I}$. Let $E = E(\mathcal{I}) \in \bR^{n \times k}$, where $k = |\mathcal{I}|$, be the matrix whose columns are the Euclidean basis vectors $\mathbf{e}_i \in \bR^n$, $i \in \mathcal{I}(L)$, arranged in increasing order. By construction, 
\begin{equation} \label{eqn:LII}
L_{[\cI, \cI]} = E^{\intercal} L E \in \mathscr{L}_{++}(k)
\end{equation}
is invertible.

\begin{definition}[Chronological inverse] \label{def:chroninv}
The chronological inverse of $L = \cC_{\min}(A)$, $A \in \mathscr{S}_+(n)$, is defined in terms of the notation above by
\begin{equation}\label{eq:chroninv}
L^{\ominus}:=E (L_{[\cI,\cI]})^{-1} E^{\intercal} \in \mathscr{L}_+(n).
\end{equation}
\end{definition}

If $L \in \mathscr{L}_{++}(n)$, then $L^{\ominus}$ coincides with the usual inverse $L^{-1}$. In general, the chronological inverse differs from the Moore--Penrose inverse but is more suitable for our purposes.

\begin{theorem}[$\AW_2$ between Gaussian distributions] \label{thm:AW2.Gaussian.distributions} { \ }
\begin{itemize}
\item[(i)] For any Gaussian distribution $\mu = \cN(a, A)$, where $(a, A) \in \bR^{Nd} \times \mathscr{S}_+(Nd)$, we have
\begin{equation} \label{eqn:Gaussian.claim}
\AW_2(\bG^{a, \cC_{\min}(A)}, \bX^{\mu}) = 0,
\end{equation}
where $\bX^{\mu}$ is the canonical representation given in Definition \ref{def:embedding}.
\item[(ii)] For $\mu = \cN_{Nd}(a, A)$ and $\nu = \cN_{Nd}(b, B)$, we have
\begin{equation} \label{eqn:AW2.Gaussian.distributions}
\AW_2^2(\mu, \nu) =  \|a - b\|_2^2 + \dist_{\AW}^2(\cC_{\min}(A), \cC_{\min}(B)).
\end{equation}
\end{itemize}
\end{theorem}
\begin{proof}
(i) Let $L = \cC_{\min}(A)$ and consider the filtered Gaussian process
\[
\bG^{a, L} = (\bR^{Nd}, \cB(\bR^{Nd}), \cN_{Nd}(0, I), (\cF_t)_{t = 1}^{N}, X = a + L\epsilon),
\]
where $\epsilon$ is the canonical process on the path space $\bR^{Nd}$. 

As an intermediate step, consider the filtered process
\begin{equation} \label{eqn:X.tilde}
\tilde{\bG}^{a, L} := (\bR^{Nd}, \cB(\bR^{Nd}), \tilde{\bP}, (\tilde{\cF}_t)_{t = 1}^{N}, \tilde{X} = a + L\tilde{\epsilon}),
\end{equation}
where $\tilde{\epsilon} = (\tilde{\epsilon}_1, \ldots, \tilde{\epsilon}_N) = (\tilde{\epsilon}_{[1]}, \ldots, \tilde{\epsilon}_{[Nd]})$ is defined by
\begin{equation} \label{eqn:epsilon.tilde}
\tilde{\epsilon}_{[i]} := \left\{\begin{array}{ll}
        \epsilon_{[i]}, & \text{if } L_{[i,i]} > 0;\\
        0, & \text{if } L_{[i,i]} = 0,
\end{array}\right.
\end{equation}
$\tilde{\cF}_t := \sigma(\tilde{\epsilon}_1, \ldots, \tilde{\epsilon}_t)$ is the filtration induced by $\tilde{\epsilon}$, and $\tilde{\bP} := \cL(\tilde{\epsilon})$ is the law of $\tilde{\epsilon}$. That is, all components of the canonical process $\epsilon$ that are not active in the product $L\epsilon$ (due to the zero-column condition in Proposition \ref{prop:Chronological.Cholesky}) are replaced by zeros.

Consider the coupling $\pi = \cL(\epsilon, \tilde{\epsilon})$ (the joint distribution of $\epsilon$ and $\tilde{\epsilon}$) defined by letting $\epsilon \sim \cN_{Nd}(0, I)$ and then defining $\tilde{\epsilon}$ by \eqref{eqn:epsilon.tilde}. Clearly, $\pi$ is a bicausal coupling between $\bG^{a, L}$ and $\tilde{\bG}^{a, L}$, under which $X = \tilde{X}$ almost surely. It follows that $\AW_2(\bG^{a, L}, \tilde{\bG}^{a, L}) = 0$.

Next, we prove that $\AW_2(\tilde{\bG}^{a, L}, \bX^{\mu}) = 0$. Then, from the triangular inequality we have $\AW_2(\bG^{a, L}, \bX^{\mu}) = 0$. To this end, consider the process $\tilde{X}$ in \eqref{eqn:X.tilde}. Observe  that for any $i \in [Nd]$, we have
\begin{equation} \label{eqn:filtration.claim}
\sigma(\tilde{X}_{[1]}, \ldots, \tilde{X}_{[i]}) = \sigma(\tilde{\epsilon}_{[1]}, \ldots, \tilde{\epsilon}_{[i]}).
\end{equation}
Indeed, since $\tilde{X} = a + L\tilde{\epsilon}$ and $L$ is lower triangular, it is clear that $\sigma(\tilde{X}_{[1]}, \ldots, \tilde{X}_{[i]}) \subset \sigma(\tilde{\epsilon}_{[1]}, \ldots, \tilde{\epsilon}_{[i]})$. On the other hand, we have
\[
\tilde{\epsilon} = L^{\ominus} (\tilde X - a),
\]
where $L^{\ominus}$ is the chronological inverse of $L$, and the reverse inclusion follows. In particular, letting $i = d, 2d, \ldots, Nd$, we have
\begin{equation} \label{eqn:filtration.equal}
\cF_t^{\tilde{X}} := \sigma(\tilde{X}_1, \ldots, \tilde{X}_t) = \sigma(\tilde{\epsilon}_1, \ldots, \tilde{\epsilon}_t) = \tilde{\cF}_t,
\end{equation}
so $\tilde{X}$ and $\tilde{\epsilon}$ induce the same filtration.




From \eqref{eqn:filtration.equal}, $\tilde{\pi} := \mathcal{L}(\tilde{\epsilon}, \tilde{X})$ is a bicausal coupling between $\tilde{\bG}^{a, L}$ and $\bX^{\mu}$ with zero transport cost. It follows that $\AW_2(\tilde{\bG}^{a, L}, \bX^{\mu}) = 0$.

(ii) By (i), we have
\[
\AW_2(\bX^{\mu}, \bX^{\nu}) = \AW_2(\bG^{a, \cC_{\min}(A)}, \bG^{b, \cC_{\min}(B)}).
\]
We obtain \eqref{eqn:AW2.Gaussian.distributions} from Theorem \ref{thm:AW.filtered.Gaussians}.
\end{proof}

\begin{corollary}[$\AW_2$-completion of Gaussian distributions] \label{cor:law.completion}
Given $N$ and $d$, consider the set
\[
\mathscr{N}(N, d) := \{ \cN_{Nd}(a, A): a \in \bR^{Nd}, A \in \mathscr{S}_+(Nd)\}
\]
of Gaussian distributions on the path space $\bR^{Nd}$. Then, the completion of $(\{ [\bX^{\mu}] : \mu \in \mathscr N(N,d) \}, \AW_2)$ is $(\FG, \AW_2)$.
\end{corollary}

\begin{proof}
From Theorem \ref{thm:AW2.Gaussian.distributions}(i), if $\mu = \cN_{Nd}(a, A)$, then
\[
[\bX^{\mu}] =[ \bG^{a, \cC_{\min}(A)}].
\]
(That is, $\AW_2(\bX^{\mu}, \bG^{a, \cC_{\min}(A)}) = 0$.) Thus, we may regard $\{ [\bX^{\mu}] : \mu \in \mathscr{N}(N, d) \}$ as a subspace of $(\FG, \AW_2)$.

From Theorem \ref{thm:completeness}, $\mathscr{L}_{++}(N, d) / \sim$ is dense in $\tilde{\mathscr{L}}(N, d)$ with respect to $\dist_{\AW}$. This (together with Theorem \ref{thm:AW.filtered.Gaussians}) implies that 
\[
\{[\bX^{\mu}] : \mu = \cN_{Nd}(a, A), a \in \bR^{Nd}, A \in \mathscr{S}_{++}(Nd)\}
\]
is dense in $\FG$. The proof is complete by noting that the metric space $(\FG, \AW_2)$ is complete by Theorem \ref{thm:completeness}. 
\end{proof}

We end this section by noting an optimality property of the minimal Cholesky factor.

\begin{proposition}[Optimality of minimal Cholesky factor] \label{prop:optimality}
Let $A \in \mathscr{S}_+(Nd)$. Then the minimum
\[
\min_{L \in \mathscr{L}(N, d): LL^{\intercal} = A} \AW_2( \bG^{L}, \bG^{I}) = \min_{L \in \mathscr{L}(N, d): LL^{\intercal} = A}  \dist_{\mathrm{AW}}(L, I)
\]
is attained by $L = \cC_{\min}(A)$.
\end{proposition}
\begin{proof}
Let $L \in \mathscr{L}(N, d)$ be such that $LL^{\intercal} = A$. Consider the filtered process $\bG^{L}$ where $X = L\epsilon$. Fix $t \in [N]$ and write
\[
X_t = U_t + V_t, \quad \text{where } U_t := L_{t,t} \epsilon_t \text{ and } V_t := \sum_{s < t} L_{t,s} \epsilon_s.
\]
Consider the conditional covariance matrix with respect to $\cF_{t-1}^X = \sigma(X_s: s < t)$. Since $\epsilon_t$ is independent of $(\epsilon_s)_{s < t}$ and hence $(X_s)_{s < t}$, we have
\begin{equation} \label{eqn:conditional.variance}
\begin{split}
\Sigma_t &:= \Var(X_t \mid \cF_{t-1}^X) \\
&= \Var(U_t \mid \cF_{t-1}^X) + \Var(V_t \mid \cF_{t-1}^X) + 2\,\Cov(U_t,V_t \mid \cF_{t-1}^X)      \\
&= \Var(U_t) + \Var(V_t \mid \cF_{t-1}^X) \\
&= L_{t,t} L_{t,t}^{\intercal} + \Var(V_t \mid \cF_{t-1}^X)
\end{split}
\end{equation}
Since $U_t$ is independent of $(V_t,\cF_{t-1}^X)$, the conditional covariance term is zero. Thus $L_{t,t} L_{t,t}^{\intercal} \preceq \Sigma_t$ in Loewner partial order. By \cite[Corollary 7.7.4]{HornJohnson2013}, we have
\[
\| L_{t,t}^{\intercal} \|_* = \tr \big( (L_{t,t} L_{t,t}^{\intercal})^{\frac{1}{2}} \big) \leq \tr(\Sigma_t^{\frac{1}{2}}).
\]
It follows from Theorem \ref{thm:AW.filtered.Gaussians} that
\begin{equation} \label{eqn:optimality.lower.bound}
\AW_2^2(\bG^{L}, \bG^{I}) \geq \tr(A) + Nd - 2\sum_{t = 1}^N \tr(\Sigma_t^{\frac{1}{2}}).
\end{equation}

Now let $L = \cC_{\min}(A)$. Instead of $\bG^{L}$, consider the filtered process $\tilde{\bG}^{L} := \tilde{\bG}^{0, L}$ defined by \eqref{eqn:X.tilde}, where $\tilde{X} = L \tilde{\epsilon}$. Since $\cF^{\tilde{X}}_{t-1} = \cF^{\tilde{\epsilon}}_{t-1}$, the argument in \eqref{eqn:conditional.variance} shows that
\[
L_{t,t} L_{t,t}^{\intercal} = \Var(\tilde{X}_t \mid \cF_{t-1}^{\tilde{X}}) = \Sigma_t,
\]
where the last equality holds since $X$ and $\tilde{X}$ have the same Gaussian distribution $\cN(0, A)$. It follows that $L = \cC_{\min}(A)$ attains the lower bound in \eqref{eqn:optimality.lower.bound}, and the proposition is proved.
\end{proof}

\begin{example}
Consider again the matrix 
\[
A = \begin{bmatrix} 0 & 0 \\ 0 & 1 \end{bmatrix} \in \mathscr{S}_+(2)
\]
as in Example \ref{eg:minimal.Choleky}. Here $\cC_{\min}(A) = A$. We have
\[
\dist_{\AW}^2(\cC_{\min}(A), I) = 1 + 2 - 2(0 + 1) = 1.
\]
For $L(\theta)$ given by \eqref{eqn:L.theta}, $L(\theta)L(\theta)^{\intercal} = A$. We have
\[
\dist_{\AW}^2(L(\theta), I) = 1 + 2 - 2(0 + |\sin \theta|) = 3 - 2 \sin \theta \geq 1.
\]
\end{example}

On the other hand, for given $A, B \in \mathscr{S}_+(Nd)$, it is generally not true that 
\begin{equation*}
\begin{split}
&\min \{ \dist_{\mathrm{AW}}(L, M) : L, M \in \mathscr{L}_+(Nd), LL^{\intercal} = A, MM^{\intercal} = B \} 
\end{split}
\end{equation*}
is attained by $(L, M) = (\cC_{\min}(A), \cC_{\min}(B))$. Here is a simple counterexample.

\begin{example}
Let $N = 2$, $d = 1$, and consider
\[
A = \begin{bmatrix} 1 & 1 \\ 1 & 1 \end{bmatrix}, \quad B = \begin{bmatrix} 0 & 0 \\ 0 & 1 \end{bmatrix}.
\]
We have
\[
\cC_{\min}(A) = \begin{bmatrix} 1 & 0 \\ 1 & 0 \end{bmatrix}, \quad \cC_{\min}(B) = B.
\]
It follows that 
\[
\dist_{\AW}^2( \cC_{\min}(A), \cC_{\min}(B)) = 2 + 1 - 2 \cdot 0 = 3.
\]
On the other hand, for
\[
L = \cC_{\min}(A), \quad M = \begin{bmatrix} 0 & 0 \\ 1 & 0 \end{bmatrix},
\]
we have $MM^{\intercal} = B$ and
\[
\dist_{\mathrm{AW}}^2(L, M) = 2 + 1 - 2 \cdot 1 = 1,
\]
which is strictly smaller than $\dist_{\mathrm{AW}}^2( \cC_{\min}(A), \cC_{\min}(B))$. Intuitively, this is because this choice of $M$ allows us to couple the noises at time $1$.
\end{example}

\section{Gaussian martingales and Markov processes, and projections} \label{sec:mgle.Markov}
In this section, we study the subspaces of martingales and Markov processes within the space of filtered Gaussian processes. As before, we use the filtration $\bF$ induced by the noise process $\epsilon$. 

\begin{definition}[Filtered Gaussian martingales and Markov processes] { \ }
\begin{itemize}
\item[(i)] $\cFG_{\mathrm{mgle}}$ is the subset of $\bG^{a, L} \in \cFG$ such that $X$ is an $\bF$-martingale.
\item[(ii)] $\cFG_{\mathrm{Markov}}$ is the subset of $\bG^{a, L} \in \cFG$ such that $X$ is $\bF$-Markovian.\footnote{By definition, this means that $\bE[ \varphi(X_s) \mid \cF_t] = \bE[\varphi(X_s) \mid X_t]$ for all $s \geq t$ and bounded measurable function $\varphi$.}
\end{itemize}
We let $\FG_{\mathrm{mgle}} := \cFG_{\mathrm{mgle}} / \sim$ and $\FG_{\mathrm{Markov}} := \cFG_{\mathrm{Markov}} / \sim$ be the corresponding spaces of equivalence classes in $\FP_2$.
\end{definition}

The martingale and Markov properties can be characterized in terms of the Cholesky factor. In fact, in this context the martingale property implies the Markov property. The proof, which follows from the definition, is left to the reader.

\begin{proposition}\label{prop:martingale.Markov}
Let $\bG^{a, L} \in \cFG$, where $a \in \bR^{Nd}$ and $L \in \mathscr{L}(N, d)$.
\begin{itemize}
\item[(i)] $\bG^{a,L} \in \cFG_{\mgle}$ if and only if $a_1 = \cdots = a_N$ and
\[
L_{t, t} = L_{t+1, t} = \cdots = L_{N, t}, \quad t \in [N].
\]
That is,
\begin{equation*}
X_{t+1} = X_t + L_{t+1,t+1}\epsilon_{t+1}, \quad t \in [N-1].
\end{equation*}
\item[(ii)] $\bG^{a,L} \in \cFG_{\Markov}$ if and only if for each $t \in [N-1]$ there exists $\Phi_t \in \bR^{d \times d}$ such that 
\begin{equation} \label{eqn:Markov.transition.matrix}
\begin{bmatrix}
L_{t+1,1} & \cdots & L_{t+1,t}
\end{bmatrix} = \Phi_t \begin{bmatrix} L_{t,1} & \cdots & L_{t,t} \end{bmatrix}.
\end{equation}
That is,
\[
X_{t+1} = (a_{t+1} - \Phi_t a_t) + \Phi_t X_t  + L_{t+1,t+1} \epsilon_{t+1}, \quad t \in [N-1].
\]
\end{itemize}
In particular, we have $\cFG_{\mgle} \subset \cFG_{\Markov}$. We let
\begin{equation*}
\mathscr{L}_{\mgle}(N, d) := \{L \in \mathscr{L}(N, d)
: \bG^{L} \in \cFG_{\mgle}\}
\end{equation*}
be the set of Cholesky factors of Gaussian martingales. Similarly, we let
\[
\mathscr{L}_{\mathrm{Markov}}(N, d) := \{L \in \mathscr{L}(N, d)
: \bG^{L} \in \cFG_{\mathrm{Markov}}\}.
\]
\end{proposition}

From \eqref{eqn:Markov.transition.matrix}, to specify an element $L$ of $\mathscr{L}_{\mgle}(N, d)$ we only need to specify its diagonal blocks $L_{t,t}$.

\begin{proposition}[Weak geodesic convexity of $\FG_{\mgle}$] \label{prop:mgle.Markov}
Let $\bX_0 = \bG^{a_0, L_0}, \bX_1 = \bG^{a_1, L_1} \in \cFG_{\mgle}$, and let $([\hat{\bX}_u])_{u\in[0, 1]} \subset \FP_2$ be a geodesic from $[\bX_0]$ to $[\bX_1]$ constructed as in Proposition \ref{prop:geodesic.convexity}. Then $([\hat{\bX}_u])_{u\in[0, 1]} \subset \FG_{\mgle}$.
\end{proposition}
\begin{proof}
Using the notation in Proposition \ref{prop:geodesic.convexity}, write $\hat{\bX}_u = \bG^{\hat{a}_u, \hat{L}_u}$, where $\hat{a}_u = (1 - u) a_0 + ua_1$ and 
\[
\hat{L}_u = L_0 + u (L_1 Q - L_0) = (1 - u) L_0 + u L_1 Q, 
\]
where $Q = \diag(Q_1, \ldots, Q_N) \in \mathscr{Q}(L_0, L_1)$. From Proposition \ref{prop:martingale.Markov}, we have $(\hat{a}_u)_1 = \cdots = (\hat{a}_u)_N$. Moreover, for $s \geq t$, we have
\begin{equation*}
\begin{split}
(\hat{L}_u)_{s, t} &= (1 - u) (L_0)_{s, t} + u (L_1)_{s, t} Q_t \\
  &= (1 - u) (L_0)_{t,t} + u (L_1)_{t,t} Q_t = (\hat{L}_u)_{t, t}.
\end{split}
\end{equation*}
Using Proposition \ref{prop:martingale.Markov} again, we conclude that $\hat{\bX}_u \in \mathcal{FG}_{\mathrm{mgle}}$.
\end{proof}

On the other hand, $\FG_{\Markov}$ does not satisfy the property in Proposition \ref{prop:mgle.Markov} as shown by the following example.

\begin{example}\label{ex:geodesic_not_markov}
Let $N=3$ and $d=1$.
Consider $\bX_0 = \bG^{L_0}$ and $\bX_1 = \bG^{L_1}$ where
\[
L_0=\begin{bmatrix}1&0&0\\0&1&0\\0&0&1\end{bmatrix}\quad \text{and} \quad
L_1=\begin{bmatrix}1&0&0\\1&1&0\\1&1&1\end{bmatrix}.
\]
From Proposition \ref{prop:martingale.Markov}, we have $\bX_0, \bX_1 \in \cFG_{\mgle} \subset \cFG_{\Markov}$: the former is a Gaussian white noise, the latter is a Gaussian random walk. We may verify that $Q := I \in \mathscr{Q}(L_0, L_1)$.  Consider the geodesic $([\bX_u])_{u \in [0, 1]}$ where $\bX_u = \bG^{L_u}$ and
\[
L_u=(1-u)L_0+uL_1=\begin{bmatrix}1&0&0\\u&1&0\\u&u&1\end{bmatrix}.
\]
From Proposition \ref{prop:martingale.Markov}(ii), we see that $\bX_u \notin \cFG_{\Markov}$ for $u \in (0, 1)$. 
\end{example}

For Gaussian martingales, the adapted Wasserstein distance can be expressed in terms of the Bures--Wasserstein distance.

\begin{corollary}[$\AW_2$ between Gaussian martingales] \label{cor:AW2.mgle}
Let $L, M \in \mathscr{L}(N, d)$ be such that $\bG^{L}, \bG^{M} \in \cFG_{\mgle}$. Then
\begin{equation} \label{eqn:AW2.mgle}
\dist_\mathrm{AW}^2(L,M)=\sum_{t=1}^N (N-t+1)\dist_\BW^2(L_{t,t}L_{t,t}^\intercal,M_{t,t}M_{t,t}^\intercal).
\end{equation}
\end{corollary}
\begin{proof}
By \eqref{eqn:dist.AW.alternative}, we have
\begin{equation*}
\begin{split}
\dist_{\mathrm{AW}}^2(L, M) = \sum_{t = 1}^N \left( \|L_{\cdot, t}\|_{\mathrm{F}}^2 + \|M_{\cdot, t}\|_{\mathrm{F}}^2 - 2 \| (L_{\cdot,t})^{\intercal} M_{\cdot, t} \|_*\right).
\end{split}
\end{equation*}
By Proposition \ref{prop:martingale.Markov}(i), for each $t$ we have $L_{t,t} = \cdots = L_{N,t}$ and $M_{t,t} = \cdots = M_{N, t}$. This allows us to simplify the above to get
\begin{equation*}
\begin{split}
\dist_{\mathrm{AW}}^2(L, M) &= \sum_{t = 1}^N (N - t + 1) \left( \|L_{t,t}\|_{\mathrm{F}}^2 + \|M_{t,t}\|_{\mathrm{F}}^2 - 2\| (L_{t,t})^{\intercal} M_{t,t} \|_* \right) \\
&= \sum_{t = 1}^N (N - t + 1) \dist_{\mathrm{BW}}^2 (L_{t,t} L_{t,t}^{\intercal}, M_{t,t} M_{t,t}^{\intercal}),
\end{split}
\end{equation*}
where in the last equality we apply Proposition \ref{prop:Wasserstein.Cholesky}.
\end{proof}

Next, we use the Procrustes representation (Theorem \ref{thm:Procrustes}) of $\dist_{\mathrm{AW}}$ to study a basic projection problem in $(\mathscr{L}(N, d), \dist_{\mathrm{AW}})$. Given a non-empty set $\mathscr{K} \subset \mathscr{L}(N, d)$ and $L \in \mathscr{L}(N, d)$, consider
\begin{equation} \label{eqn:projection}
\inf_{M \in \mathscr{K}} \dist_{\mathrm{AW}}(L, M).
\end{equation}

\begin{assumption} \label{ass:projection}
The set $\mathscr{K}$ satisfies the following conditions:
\begin{itemize}
\item[(i)] $\mathscr{K}$ is {\it orthogonally right-invariant} in the sense that
\[
\mathscr{K} = \mathscr{K} \mathscr{O}(N, d) := \{K Q: Q \in \mathscr{O}(N, d)\}.
\]
\item[(ii)] $\mathscr{K}$ is closed with respect to the Frobenius norm.
\item[(iii)] $\mathscr{K}$ is convex in the usual sense.
\end{itemize}
\end{assumption}

Some remarks are in order. By Proposition \ref{prop:dist.AW.degeneracy}, orthogonal right-invariance means that $\mathscr{K}$ contains all equivalence classes of elements of $\mathscr{K}$, that is, if $L \in \mathscr{K}$ then $[L] \subset \mathscr{K}$. This ensures that \eqref{eqn:projection} is equivalent to the projection problem in the space $\tilde{\mathscr{L}}(N, d) = \mathscr{L}(N, d) / \sim$ of equivalence classes. Since convergence in Frobenius norm is equivalent to componentwise convergence of matrix entries, (ii) is straightforward to check in practice. From \eqref{eqn:Frobenius.bound}, this implies closedness with respect to $\dist_{\mathrm{AW}}$. Finally, convexity is needed for the projection to be well-posed. 

\begin{proposition}[Projection] \label{prop:right.invariance}
Suppose that $\mathscr{K} \subset \mathscr{L}(N, d)$ is nonempty and satisfies Assumption \ref{ass:projection}. Then, for any $L \in \mathscr{L}(N, d)$ there exists $\hat{L} \in \mathscr{K}$ such that
\begin{equation} \label{eqn:projection2}
\dist_{\mathrm{AW}}(L, \hat{L}) = \inf_{M \in \mathscr{K}} \dist_{\mathrm{AW}}(L, M).
\end{equation}
In fact, we have
\begin{equation} \label{eqn:projection.argmin}
\argmin_{M \in \mathscr{K}} \dist_{\mathrm{AW}}(L, M) = \left[ \argmin_{M \in \mathscr{K}} \|L - M\|_{\mathrm{F}} \right],
\end{equation}
where $[\cdot]$ denotes the equivalence class in $\dist_{\mathrm{AW}}$. In particular, the solution in $\tilde{\mathscr{L}}(N, d)$ is unique.
\end{proposition}
\begin{proof}
By Theorem \ref{thm:Procrustes} and right-invariance, we have
\begin{equation} \label{eqn:projection.proof}
\begin{split}
\inf_{M \in \mathscr{K}} \dist_{\mathrm{AW}}(L, M) &= \inf_{M \in \mathscr{K}} \inf_{Q \in \mathscr{O}(N, d)} \|L - MQ\|_{\mathrm{F}} = \inf_{\tilde{M} \in \mathscr{K}} \|L - \tilde{M}\|_{\mathrm{F}}.
\end{split}
\end{equation}
From Assumption \ref{ass:projection} and the Hilbertian nature of $\|\cdot\|_{\mathrm{F}}$, there exists a unique $\tilde{L} \in \mathscr{K}$ such that 
\[
\|L - \tilde{L}\|_{\mathrm{F}} = \inf_{\tilde{M} \in \mathscr{K}} \|L - \tilde{M}\|_{\mathrm{F}}.
\]
From \eqref{eqn:projection.proof}, we have
\[
\dist_{\mathrm{AW}}(L, \tilde{L}) \leq \|L - \tilde{L}\|_{\mathrm{F}} = \inf_{M \in \mathscr{K}} \dist_{\mathrm{AW}} (L, M).
\]
It follows that $\tilde{L}$ is optimal for \eqref{eqn:projection2}. By right-invariance and Proposition \ref{prop:dist.AW.degeneracy}, the set of optimizers is precisely $[\tilde{L}]$, the equivalence class of $\tilde L$.
\end{proof}

In the following, we specialize to the case $\mathscr{K} = \mathscr{L}_{\mgle}(N, d)$, and show that the Gaussian martingale projection has a simple explicit expression. 

\begin{corollary}[Martingale projection]\label{cor:mgle.proj}
Let $L \in \mathscr{L}(N, d)$. Then $\hat{L} \in \mathscr{L}_{\mgle}(N, d)$ is optimal for
\[
\inf_{M \in \mathscr{L}_{\mgle}(N, d)} \dist_{\mathrm{AW}}(L, M)
\]
if and only if it has the form
\[
\hat{L}_{t,t} = \left( \frac{1}{N - t + 1} \sum_{s = t}^N L_{s,t}\right) Q_t
\]
where $Q = \diag(Q_1, \ldots, Q_N) \in \mathscr{O}(N, d)$.
\end{corollary}
\begin{proof}
From Proposition \ref{prop:martingale.Markov}, we see that $\mathscr{L}_{\mgle}(N, d)$ satisfies Assumption \ref{ass:projection}, and hence Proposition \ref{prop:right.invariance} applies. For $\tilde{L} \in \mathscr{L}_{\mgle}(N, d)$, we have
\begin{equation}
\|L - \tilde{L}\|_{\mathrm{F}}^2 = \sum_{t = 1}^N \sum_{s = t}^N \|L_{s,t} - \tilde{L}_{t,t}\|_{\mathrm{F}}^2.
\end{equation}
To minimize this quantity, we pick
\[
\tilde{L}_{t,t} = \frac{1}{N - t + 1} \sum_{s = t}^N L_{s,t},
\]
and the rest follows from \eqref{eqn:projection.argmin}.
\end{proof}

On the other hand, as the following example shows, $\mathscr{L}_{\mathrm{Markov}}$ is not closed in Frobenius norm.

\begin{example}\label{ex:markov.closed}
Let $N = 2$. For $n \geq 1$, let
\[
L^{(n)}=\begin{bmatrix}
    \frac{1}{n}I & 0\\
    I & I
\end{bmatrix} \in \mathscr{L}_{\mathrm{Markov}}(N, d).
\]
We have
\[
L^{(n)} \rightarrow L = \begin{bmatrix} 0 & 0\\
I & I
\end{bmatrix}
\]
in Frobenius norm but $L \notin \mathscr{L}_{\mathrm{Markov}}(N, d)$.
\end{example}


\begin{remark}
Corollary \ref{cor:mgle.proj} for Gaussian martingales can be generalized to a class of Gaussian Markov processes that we term common-dynamics Markov processes. Let $\Phi = (\Phi_t)_{t = 1}^{N-1}$ be a fixed sequence in $\bR^{d \times d}$. Let
\[
\mathscr{L}_{\Phi}(N, d) = \{ L \in \mathscr{L}(N, d): L_{t+1,s} = \Phi_t L_{t,s}, \quad t \in [N-1], s \leq t\}
\]
Probabilistically, this means that $X_{t+1} = \Phi_t X_t + L_{t+1,t+1} \epsilon_{t+1}$ is an autoregressive process. Here, the transition matrices $\Phi_t$ are fixed but the volatility of the innovation is allowed to vary. Letting $\Phi_t = I$ for all $t$ recovers $\mathscr{L}_{\mgle}(N, d)$. It is straightforward to check that $\mathscr{L}_{\Phi}(N, d)$ satisfies Assumption \ref{ass:projection} and the projection can be explicitly computed.
\end{remark}

\section{Adapted Brenier coupling} \label{sec:adapted.Brenier}
For general filtered processes or distributions, the adapted optimal transport problem \eqref{eqn:AW2} (or \eqref{eqn:AW2.laws}) is difficult to solve. Beyond the Gaussian setting considered here, only few explicit solutions (and expressions of the optimal transport cost) have been found. Thus, it is useful to have ``off-the-shelf'' bicausal couplings that can be readily applied, even if they may be suboptimal. For example, the Knothe--Rosenblatt coupling defines a bicausal coupling between laws of univariate processes, and is $\AW_2$-optimal under suitable conditions on the marginals \cite{BBYZ2017, BPP23}. In this section, we study the {\it adapted Brenier coupling} which extends the Knothe--Rosenblatt coupling to multivariate processes.

We first define the adapted Brenier coupling between laws of stochastic processes. This definition has appeared in \cite[page 5]{BPP23}.\footnote{The authors of \cite{BPP23} also considered another generalization of the Knothe--Rosenblatt coupling based on the quantile process. We leave this to further study.} Then, we tailor the definition to filtered Gaussian processes. 

\begin{definition}[Adapted Brenier coupling between elements of $\mathcal{P}_2(\R^{Nd})$] \label{def:adapted.Brenier}
Let $\mu, \nu \in \cP_2(\bR^{Nd})$ be laws of $d$-dimensional stochastic processes. We say that $\pi \in \Cpl(\mu, \nu)$ is an adapted Brenier coupling of $(\mu, \nu)$ if for all $t \in [N]$ and ($\pi$-almost) all $x_{1:(t-1)}, y_{1:(t-1)} \in \bR^{(t-1)d}$ (empty when $t = 1$), the conditional distribution
\begin{equation} \label{eqn:conditional.coupling}
\pi_{t}(\dd x_{t}, \dd y_{t} \mid x_{1:(t-1)}, y_{1:(t-1)})
\end{equation}
is a Brenier coupling between the conditional marginals $\mu_{t}(\dd x_{t} \mid x_{1:(t-1)})$ and $\nu_{t}(\dd y_{t} \mid y_{1:(t-1)})$.
\end{definition}

By \cite[Proposition 5.1]{BBYZ2017}, any adapted Brenier coupling is bicausal. When $d = 1$, the comonotonic coupling between $\mu_{t}(\dd x_{t} \mid x_{1:(t-1)})$ and $\nu_{t}(\dd y_{t} \mid y_{1:(t-1)})$ is a Brenier coupling between the conditional marginals. With this choice, the adapted Brenier coupling coincides with the classical Knothe--Rosenblatt coupling.


Now let $\bX = \bG^{a, L}$ and $\bY = \bG^{b, M}$ be filtered Gaussian processes. If we only consider the distributions of the processes $X = a + L\epsilon^X$ and $Y = b + M \epsilon^Y$, namely $\mu = \cN_{Nd}(a, A = LL^{\intercal})$ and $\nu = \cN_{Nd}(b, B = MM^{\intercal})$, we may construct an adapted Brenier coupling, according to Definition \ref{def:adapted.Brenier}, by optimally coupling the conditional marginals which are Gaussian. Specifically, the conditional marginals can be expressed in terms of the chronological inverses of the minimal Cholesky factors of $A$ and $B$. 

On the other hand, a bicausal coupling between $\bX$ and $\bY$ is a coupling between the noises $\epsilon^X$ and $\epsilon^Y$. Coupling $\epsilon_{t}^X$ and $\epsilon_{t}^Y$ (conditioned on $(\epsilon_{1:(t-1)}^X, \epsilon_{1:(t-1)}^Y)$) instead of coupling $X_{t}$ and $Y_{t}$ (conditioned on $(X_{1:(t-1)}, Y_{1:(t-1)})$) leads to Definition \ref{def:adapted.Brenier.FG} below. Given $\bX = \bG^{a, L}$, we let
\[
\hat{X}_{t} = \hat{X}_{t}(\omega_{1:(t-1)}^{\bX}) := a_{t} + \sum_{s \leq t-1} L_{t,s} \omega_s^{\bX}
\]
be the conditional expectation of $X_{t}$ given $\epsilon_{1:(t-1)}^X = \omega_{1:(t-1)}^{\bX}$. It follows that $X_t = \hat{X}_{t} + L_{t,t} \omega_t^{\bX}$. We define $\hat{Y}_t = \hat{Y}_{t}(\omega_{1:(t-1)}^{\bY})$ similarly. When $\omega_{1:(t-1)}^{\bX}, \omega_{1:(t-1)}^{\bY} \in \bR^{(t-1)d}$ are given, we may regard $\hat{X}_t, \hat{Y}_t$ as constants.

\begin{definition} [Adapted Brenier coupling between elements of $\cFG$] \label{def:adapted.Brenier.FG}
Let $\bX = \bG^{a, L}$ and $\bY = \bG^{b, M}$ be elements of $\cFG$. An adapted Brenier coupling between $\bX$ and $\bY$ is an element $\pi \in \Cpl_{\mathrm{bc}}(\bX, \bY)$ such that for all $t \in [N]$ and ($\pi$-almost) all $\omega_{1:(t-1)}^{\bX}, \omega_{1:(t-1)}^{\bY} \in \bR^{(t-1)d}$, the conditional distribution
\begin{equation} \label{eqn:adapted.Brenier.conditional.distribution}
\pi_{t}(\dd \omega_{t}^{\bX}, \dd \omega_{t}^{\bY} \mid \omega^{\bX}_{1:(t-1)}, \omega^{\bY}_{1:(t-1)})
\end{equation}
of $(\epsilon_{t}^{X}, \epsilon_{t}^{Y})$ given $(\epsilon_{1:(t-1)}^X, \epsilon_{1:(t-1)}^Y) = (\omega^{\bX}_{1:(t-1)}, \omega^{\bY}_{1:(t-1)})$ solves the optimal transport problem
\begin{equation}\label{eqn:one.step.AB}
\inf_{\gamma \in \Cpl(\cN_d(0, I), \cN_d(0, I))} \int \| (\hat{X}_{t} + L_{t,t}\omega_{t}^{\bX}) - (\hat{Y}_{t} + M_{t,t}\omega_{t}^{\bY}) \|_2^2  \gamma(\dd \omega_{t}^{\bX}, \dd \omega_{t}^{\bY}).
\end{equation}
\end{definition}

Comparing the definitions of the adapted Brenier coupling and the $\AW_2$-optimal coupling (Definition \ref{def:AW2}), we see that the former asks for one-step optimality while the latter optimizes a global cost. Also, note that letting $L$ and $M$ be minimal Cholesky factors recovers the adapted Brenier coupling between Gaussian distributions in the sense of Definition \ref{def:adapted.Brenier}.

We proceed to characterize adapted Brenier couplings between filtered Gaussian processes. For readability, in the following statement we restrict to Gaussian bicausal couplings. With some more work, it can be shown that the minimal transport cost in (ii) is unchanged even if non-Gaussian adapted Brenier couplings are allowed.

\begin{theorem}[Adapted Brenier transport] \label{thm:adapted.Brenier}
Let $\bX = \bG^{a, L}$ and $\bY = \bG^{b, M}$ be elements of $\cFG$. Let $L_{t,t}^{\intercal} M_{t,t}$ have singular value decomposition $U_t \Sigma_t V_t^{\intercal}$. 
\begin{itemize}
\item[(i)] A Gaussian bicausal coupling $\pi^P$ of the form \eqref{eqn:joint.Gaussian}, where $P = \diag(P_1, \ldots, P_N)$ with $\|P_t\|_{2 \rightarrow 2} \leq 1$, is an adapted Brenier coupling of $(\bX, \bY)$ if and only if
\begin{equation} \label{eqn:adapted.Brenier.characterize} 
P_t \in \mathscr{P}( L_{t,t}^{\intercal} M_{t,t}), \quad t \in [N].
\end{equation}
Its transport cost is given by
\begin{equation} \label{eqn:adapted.Brenier.cost}
\bE_{\pi^P}[\|X - Y\|_2^2] = \|a- b\|_2^2 + \|L\|_{\mathrm{F}}^2 + \|M\|_{\mathrm{F}}^2 - 2 \sum_{t = 1}^N \tr ( (L^{\intercal} M)_{t,t} P_t ).
\end{equation}
When $L_{t,t}$ and $M_{t,t}$ are non-degenerate, the optimal $P$ is uniquely given by $P_t = V_t U_t^{\intercal}$.
\item[(ii)]  Write $U_t = \begin{bmatrix}U_{t,1} & U_{t,0}\end{bmatrix}$ and $V_t = \begin{bmatrix}V_{t,1} & V_{t,0}\end{bmatrix}$ as in Proposition \ref{prop:trace}(ii). The minimal transport cost
\begin{equation} \label{eqn:Brenier.minimal.cost}
\min_{\substack{P=\diag(P_1,\ldots,P_N)\\ P_t \in \mathscr{P}(L_{t,t}^{\intercal}M_{t,t})}}
\bE_{\pi^P}[\|X-Y\|_2^2]
\end{equation}
among all adapted Brenier couplings of the form $\pi^P$ is given by
\[
\|a - b\|_2^2 + \cD_{\mathrm{AB}}(L, M),
\]
where $\cD_{\mathrm{AB}}$ is the adapted Brenier divergence\footnote{We call this a divergence, rather than (squared) distance, because when $d \geq 2$, $\cD_{\mathrm{AB}}^{1/2}$ is not a (pseudo-)metric as can be shown by numerical examples.
} between $L$ and $M$ defined by
\begin{equation} \label{eqn:optimal.AB.cost}
\begin{split}
\cD_{\mathrm{AB}}(L, M) &:= \|L\|_{\mathrm{F}}^2 + \|M\|_{\mathrm{F}}^2 - 2 \sum_{t = 1}^N \Gamma_t(L, M), \quad \text{where}\\
\Gamma_t(L, M) &:= \max_{P_t \in \mathscr{P}( L_{t,t}^{\intercal} M_{t,t})} \tr \left( (L^{\intercal}M)_{t,t} P_t \right) \\
  &= \tr \left( U_{t,1}^{\intercal} (L^{\intercal}M)_{t,t} V_{t,1} \right) + \| U_{t,0}^{\intercal} (L^{\intercal} M)_{t,t} V_{t,0} \|_*.
\end{split}
\end{equation}
\item[(iii)] The adapted Brenier divergence admits the Procrustes representation
\begin{equation} \label{eqn:adapted.Brenier.Procrustes}
\cD_{\mathrm{AB}}(L, M) = \min_{Q \in \mathscr{O}(N, d): Q_t \in \mathscr{P}( L_{t,t}^{\intercal} M_{t,t})} \|L - MQ\|_{\mathrm{F}}^2.
\end{equation}
\end{itemize}
\end{theorem}
\begin{proof}
(i) Let $t$ and $\omega_{1:(t-1)}^{\bX}, \omega_{1:(t-1)}^{\bY} \in \bR^{(t-1)d}$ be given. The one-step optimal transport problem \eqref{eqn:one.step.AB} is equivalent to minimizing
\[
\|\hat{X}_t(\omega_{1:(t-1)}^{\bX}) - \hat{Y}_t(\omega_{1:(t-1)}^{\bY})\|_2^2 + \bE_{\gamma}[ \|L_{t,t} \epsilon_t^X - M_{t,t} \epsilon_t^Y \|_2^2],
\]
where $(\epsilon_t^X, \epsilon_t^Y) \sim \gamma$, over $\gamma \in \Cpl( \cN_d(0, I), \cN_d(0, I))$. By Proposition \ref{prop:Wasserstein.Cholesky}, $\gamma$ is optimal if and only if the (conditional) correlation matrix $\bE_{\gamma} [ \epsilon_t^Y (\epsilon_t^X)^{\intercal}]$ is an element of $\mathscr{P}( L_{t,t}^{\intercal} M_{t,t})$. Note that if $\pi = \pi^P$ then the conditional distribution \eqref{eqn:adapted.Brenier.conditional.distribution}, given by \eqref{eqn:pi.P.conditional}, has correlation matrix $P_t$. Hence $\pi^P$ is an adapted Brenier coupling if and only if \eqref{eqn:adapted.Brenier.characterize} holds. The formula \eqref{eqn:adapted.Brenier.cost} follows from \eqref{eqn:pi.P.cost}, and the form of $\mathscr{P}( L_{t,t}^{\intercal} M_{t,t})$ in the non-degenerate case follows from Proposition \ref{prop:trace}(ii).

(ii) Let $\pi^P$ be an adapted Brenier coupling. By Proposition \ref{prop:trace}(ii), for each $t$ there exists $K_t$ with $\|K_t\|_{2 \rightarrow 2} \leq 1$, such that
\[
P_t = V_{t,1} U_{t,1}^{\intercal} + V_{t,0} K_t U_{t,0}^{\intercal}.
\]
Plugging this into \eqref{eqn:adapted.Brenier.cost}, the transport cost is given by
\begin{equation} \label{eqn:transport.cost.given.K}
\begin{split}
&\|a - b\|_2^2 + \|L\|_{\mathrm{F}}^2 + \|M\|_{\mathrm{F}}^2\\
&- 2\sum_{t = 1}^N \left(\tr ((L^{\intercal}M)_{t,t} V_{t,1} U_{t,1}^{\intercal}) + \tr ( (L^{\intercal} M)_{t,t} V_{t,0} K_t U_{t,0}^{\intercal})\right).
\end{split}
\end{equation}
By invariance of the trace under cyclic permutation of the product, the first trace term is equal to
\[
\tr \left( U_{t,1}^{\intercal} (L^{\intercal}M)_{t,t} V_{t,1} \right),
\]
and the last trace term can be rewritten as
\[
\tr \left(   (U_{t,0}^{\intercal} (L^{\intercal} M)_{t,t} V_{t,0}) K_t \right).
\]
By Proposition \ref{prop:trace} again, this trace is maximized if and only if
\begin{equation} \label{eqn:optimal.Kt}
K_t \in \mathscr{P} \left( U_{t,0}^{\intercal} (L^{\intercal} M)_{t,t} V_{t,0} \right), \quad t \in [N],
\end{equation}
and the minimal transport cost is given by \eqref{eqn:optimal.AB.cost}.

(iii) This follows since the optimal $K_t$ in \eqref{eqn:optimal.Kt} can always be chosen to be orthogonal (see Remark \ref{rmk:trace}).
\end{proof}

From Theorem \ref{thm:adapted.Brenier}(i), there exists an adapted Brenier coupling that solves $\AW_2(\bX, \bY)$ if and only if
\begin{equation} \label{eqn:AB.optimal.condition}
\mathscr{P}\big(L_{t,t}^\intercal M_{t,t}\big) \cap \mathscr{P}\big((L^\intercal M)_{t,t}\big) \neq \emptyset \quad \text{for all } t \in [N].
\end{equation}
This is the case if $\bX, \bY \in \cFG_{\mgle}$.

\begin{proposition}
If $\bX = \bG^{a, L}, \bY  = \bG^{b, M}\in \cFG_{\mgle}$, then any Brenier coupling is optimal for $\AW_2(\bX, \bY)$.
\end{proposition}
\begin{proof}
Since $L_{t,t} = \cdots = L_{N,t}$ and $M_{t,t} = \cdots = M_{N,t}$, we have 
\[
(L^{\intercal}M)_{t,t} = (N - t + 1) L_{t,t}^{\intercal} M_{t,t}.
\]
The conclusion then follows from Theorem \ref{thm:AW.filtered.Gaussians} and Theorem \ref{thm:adapted.Brenier}.
\end{proof}

\section{Comparison via Gaussian random matrices} \label{sec:random.matrix} 
In this section, we consider a probabilistic analysis of various transport costs between filtered Gaussian processes, by letting the Cholesky factors $L$ and $M$ be independent random matrices with Gaussian entries. When $d$ is fixed and $N \rightarrow \infty$, we show in Theorem \ref{thm:ensemble.equiv} that the transport costs of all Gaussian bicausal couplings (with deterministic block correlations) are asymptotically equivalent, in the sense that the ratio between any two tends to $1$. Also, in Theorem \ref{thm:strict-gap} we show---as one would expect---that the Bures--Wasserstein distance is strictly smaller. While other distributional assumptions may be considered, the setting of Gaussian random matrices is a natural and convenient one where many results are readily available. It is also interesting to consider the other asymptotic regimes $d \rightarrow \infty$ (with $N$ fixed) or both $d, N \rightarrow \infty$ at suitable rates.

Throughout this section we work with the following set-up. Fix a spatial dimension $d \geq 1$. On a suitable probability space $(\Omega, \mathcal{F}, \mathbb{P})$, let $L_{s,t}, M_{s,t}$, $s \geq t$, be i.i.d.~random elements of $\bR^{d \times d}$ whose entries are i.i.d.~$\cN(0, 1)$ random variables. That is, the random variables $(L_{s,t})_{[i,j]}, (M_{s,t})_{[i,j]}$, indexed by $s, t \geq 1$, $s \geq t$ and $i, j \in [d]$, are i.i.d.~standard Gaussians. For $s < t$, we let $L_{s,t} = M_{s,t} = 0_{d \times d}$. For $N \geq 1$, let
\[
L^{(N)} := (L_{s,t})_{s,t \in [N]} \quad \text{and} \quad M^{(N)} := (M_{s,t})_{s,t \in [N]}
\]
be random elements of $\mathscr{L}(N, d)$. Using these, we define
\[
\bX^{(N)} := \bG^{L^{(N)}} \quad \text{and} \quad \bY^{(N)} := \bG^{M^{(N)}},
\]
which are random elements of $\cFG(N, d)$. The corresponding stochastic processes are denoted by $X^{(N)}$ and $Y^{(N)}$ respectively.

\begin{theorem} \label{thm:ensemble.equiv}
Consider the set-up described above. Let $P_1, P_2, \ldots$ be an arbitrary (deterministic) sequence in $\mathscr{C}(d)$. For each $N$, define
\[
P^{(N)} := \diag(P_1, \ldots, P_N) \in \bR^{Nd \times Nd}
\]
and, using the notation in Lemma \ref{lem:coupling.Gaussian}, define
\[
\pi^{(N)} := \pi^{P^{(N)}},
\]
which is a Gaussian bicausal coupling of $\bX^{(N)}$ and $\bY^{(N)}$. Consider the random variable
\[
\cT(\pi^{(N)}) := \bE_{\pi^{(N)}} [ \|X^{(N)} - Y^{(N)}\|_2^2],
\]
which is the transport cost between $\bX^{(N)}$ and $\bY^{(N)}$ under $\pi^{(N)}$. Then, we have
\begin{equation}  \label{eqn:a.s.limit}
\lim_{N \rightarrow \infty} \frac{1}{N^2} \cT(\pi^{(N)}) = d^2 \quad \text{almost surely.}
\end{equation}
In particular, we have, almost surely,
\[
\lim_{N \rightarrow \infty} \frac{\cT(\pi^{(N)})}{\AW_2^2(\bX^{(N)}, \bY^{(N)})} = 1.
\]
\end{theorem}

The proof of Theorem \ref{thm:ensemble.equiv} makes use of the following lemmas.

\begin{lemma}\label{lem:frob-ssln}
It holds almost surely that
\[
\lim_{N \rightarrow \infty} \frac{1}{N^2} \|L^{(N)}\|_{\mathrm{F}}^2 = \frac{1}{2}d^2.
\]
\end{lemma}
\begin{proof}
Note that $L^{(N)}$ contains $\binom{N+1}{2} = N(N+1)/2$ non-zero blocks that contain $d\times d$ $\cN(0, 1)$ entries. Since $\bE\big[(L_{s,t})_{[i,j]}^2\big] = 1$, the desired limit follows immediately from the strong law of large numbers.
\end{proof}

\begin{lemma}\label{lem:AWdiag}
It holds almost surely that
\[
\lim_{N \rightarrow \infty} \frac{1}{N^2}\sum_{t=1}^N \left\| \left((L^{(N)})^\intercal (M^{(N)})\right)_{t,t} \right\|_* = 0.
\]
\end{lemma}

\begin{proof}
For notational simplicity, we use the shorthand $L = L^{(N)}$ and $M = M^{(N)}$. From \eqref{eqn:matrix.norm.inequality}, we have the bound
\[
\sum_{t=1}^N \normnuc{(L^\intercal M)_{t,t}}
\le \sqrt{d}\sum_{t=1}^N \normF{(L^\intercal M)_{t,t}}.
\]
Applying the Cauchy--Schwarz inequality to the right hand side and then taking expectation, we have
\[
\bE \left[ \left( \sum_{t = 1}^N \|(L^{\intercal}M)_{t,t} \|_* \right)^4 \right] \leq d^2 N^3 \sum_{t = 1}^N \bE\left[\| (L^{\intercal}M)_{t,t}\|_{\mathrm{F}}^4  \right].
\]

We proceed to bound the fourth moment of $\|(L^{\intercal}M)_{t,t}\|_{\mathrm{F}}$. By Cauchy--Schwarz again, we have
\begin{equation} \label{eqn:CS2}
\|(L^{\intercal}M)_{t,t}\|_{\mathrm{F}}^4 = \left( \sum_{i, j = 1}^d ((L^{\intercal}M)_{t,t})_{[i,j]}^2 \right)^2 \leq d^2 \sum_{i, j = 1}^d (((L^{\intercal}M)_{t,t})_{[i,j]})^4.
\end{equation}

For $t \in [N]$ and $i, j \in [d]$, write
\[
((L^{\intercal}M)_{t,t})_{[i,j]} = \sum_{s = t}^N (L_{s,t}^\intercal M_{s,t})_{[i,j]}
\]
and note that it is the sum of $m := (N - t + 1)d$ independent products of two independent standard Gaussian random variables. Let $U_1, \ldots, U_m$ denote the products. Then $\bE[U_{\ell}] = \bE[U_{\ell}^3] = 0$, $\bE[U_{\ell}^2] = 1$ and $\bE[U_{\ell}^4] = 9$. It follows that
\begin{equation*}
\begin{split}
\bE\left[ (((L^{\intercal}M)_{t,t})_{[i,j]})^4 \right] &= \bE \left[ \left( \sum_{\ell = 1}^m  U_{\ell} \right)^4 \right]\\
&= m \bE[U_1^4] + \binom{4}{2} \binom{m}{2} \bE[U_1^2]^2 \\
&= 9m + 3m(m - 1) \\
&\leq C_d(N - t + 1)^2,
\end{split}
\end{equation*}
for some constant $C_d>0$ depending only on $d$.
Combining the previous computations, we have
\[
\bE \left[ \left( \sum_{t = 1}^N \|(L^{\intercal}M)_{t,t} \|_* \right)^4 \right] \leq C_dd^3 N^3 \sum_{t = 1}^N (N - t + 1)^2 \leq C N^6,
\]
where $C$ is a constant depending only on $d$. 

Now Markov's inequality gives, for any $\delta > 0$,
\[
\bP \left( \frac{1}{N^2} \sum_{t = 1}^N \| (L^{\intercal} M)_{t,t} \|_* \geq \delta \right) \leq \frac{C N^6}{\delta^4 N^8} = \frac{C}{\delta^4 N^2}.
\]
Since $\sum_{N \geq 1} N^{-2} < \infty$, the desired almost sure limit follows from the Borel--Cantelli lemma.
\end{proof}

\begin{proof}[Proof of Theorem \ref{thm:ensemble.equiv}]
From \eqref{eqn:pi.P.cost}, we have
\begin{equation} \label{eqn:ratio.proof}
\cT(\pi^{(N)}) = \|L^{(N)}\|_{\mathrm{F}}^2 + \|M^{(N)}\|_{\mathrm{F}}^2 - 2\sum_{t = 1}^N \tr \left((( L^{(N)})^{\intercal} M^{(N)})_{t,t} P_t \right).
\end{equation}
Moreover, Proposition \ref{prop:trace} (applied to both $C$ and $-C$) implies that for any choice of $P_t \in \mathscr{C}(d)$, we have
\[
| \tr \left((( L^{(N)})^{\intercal} M^{(N)})_{t,t} P_t \right) | \leq \| (( L^{(N)})^{\intercal} M^{(N)})_{t,t} \|_*.
\]
Divide \eqref{eqn:ratio.proof} by $N^2$ and let $N \rightarrow \infty$. Each of the first two terms tends to $\frac{1}{2}d^2$ by Lemma \ref{lem:frob-ssln}, and the last tends to $0$ by Lemma \ref{lem:AWdiag}. This gives \eqref{eqn:a.s.limit} and the proof is complete. 
\end{proof}

 In Figure \ref{fig:random-matrix-ratio-comparison}, we show empirically the convergence in Theorem \ref{thm:ensemble.equiv}, along paths of simulated pairs of Gaussian matrices $(L^{(N)}, M^{(N)})$. In particular, we show the relation among the transport costs of the synchronous, adapted-Brenier and $\AW_2$-optimal couplings. Although their transport costs are asymptotically equivalent, for finite $N$ they may give quite different results. For example, on average the adapted Brenier coupling is better than the synchronous coupling.

\begin{figure}[t]
  \centering
    \includegraphics[scale = 0.5]{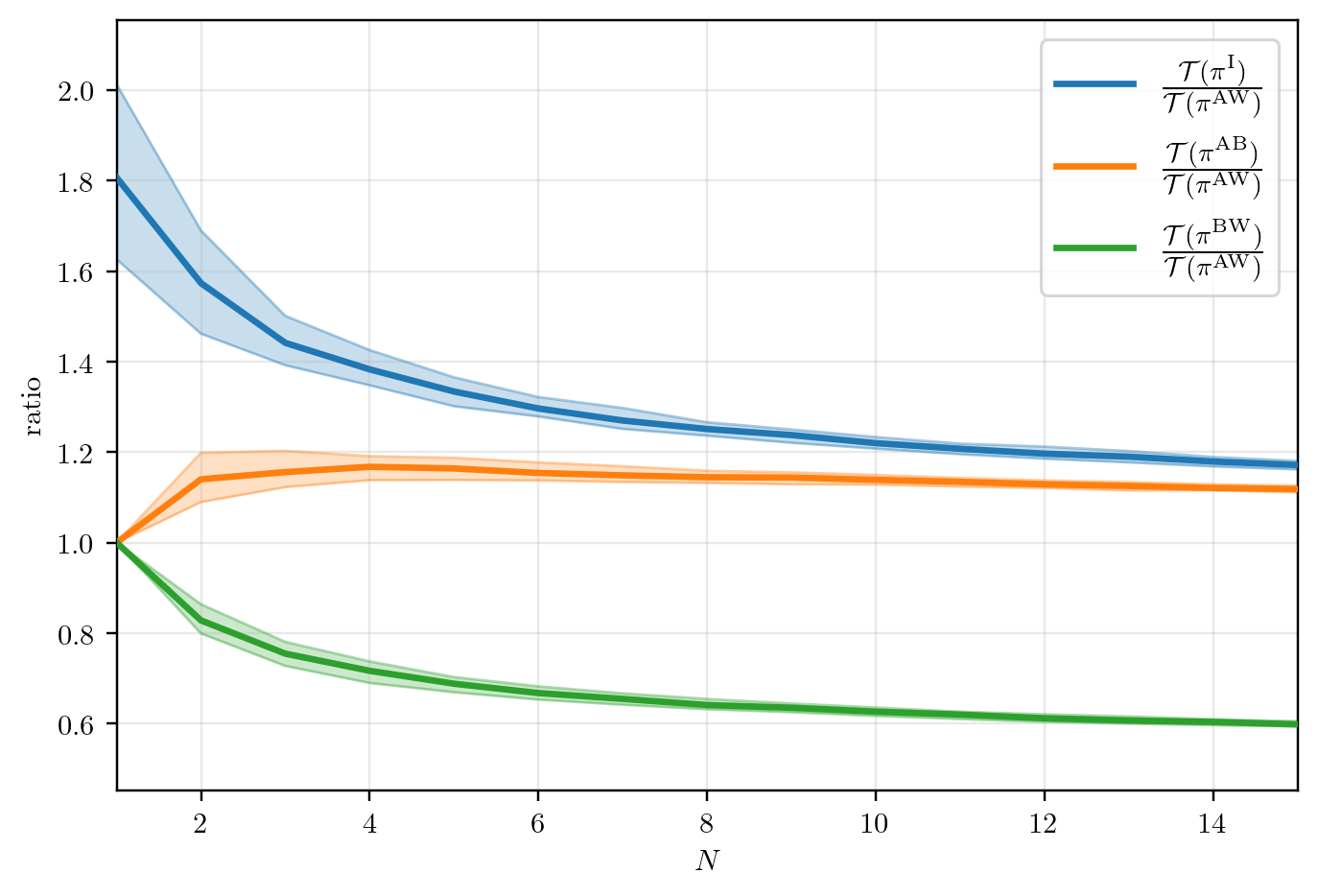}
\caption{Comparison of the transport costs $\mathcal{T}(\cdot)$ of the synchronous coupling $\pi^{I}$, adapted Brenier coupling $\pi^{\mathrm{AB}}$, $\AW_2$-optimal coupling $\pi^{\mathrm{AW}}$, and the Brenier ($\mathcal{W}_2$-optimal) coupling $\pi^{\mathrm{BW}}$. We plot the ratios along 100 simulated paths of $L^{(N)}$ and $M^{(N)}$ with $d = 5$; the solid curves show the sample medians and the shaded bands show the inter-quartile ranges.  }
  \label{fig:random-matrix-ratio-comparison}
\end{figure}

Our second result, also illustrated in Figure \ref{fig:random-matrix-ratio-comparison}, shows that the Bures--Wasserstein distance, which is not subject to the bicausality constraint, is strictly smaller.

\begin{theorem}
\label{thm:strict-gap}
Let $A^{(N)} := (L^{(N)})(L^{(N)})^{\intercal}$ and $B^{(N)} := (M^{(N)})(M^{(N)})^{\intercal}$ be the random covariance matrices induced by $L^{(N)}$ and $M^{(N)}$ respectively. Then, there exists $\rho < 1$ such that
\begin{equation} \label{eqn:strict-gap}
\limsup_{N\to\infty} \frac{\dist_{\BW}(A^{(N)}, B^{(N)})}{\dist_{\mathrm{AW}}(L^{(N)}, M^{(N)})} \le \rho \quad \text{almost surely}.
\end{equation}
\end{theorem}

We need the following lemmas whose proofs will be given in the appendix.

\begin{lemma}\label{lem:opnorm.bound}
There exists a constant $K > 0$, depending only on $d$, such that almost surely, for $N$ sufficiently large we have:
\[
\|L^{(N)}\|_{2\to 2}\le K\sqrt N \quad \text{and} \quad \|M^{(N)}\|_{2\to 2}\le K\sqrt N.
\]
\end{lemma}

\begin{lemma}\label{lem:LTM.F.limit}
It holds almost surely that
\[
\lim_{N\to\infty}\frac{1}{N^3}\|(L^{(N)})^{\intercal}M^{(N)}\|_{\mathrm{F}}^2
=
\frac{d^3}{3}
\qquad\text{almost surely.}
\]
\end{lemma}

\begin{proof}[Proof of Theorem \ref{thm:strict-gap}]
Recall from Proposition \ref{prop:Wasserstein.Cholesky} that
\[
\dist_{\mathrm{BW}}^2(A^{(N)}, B^{(N)}) = \|L^{(N)}\|_{\mathrm{F}}^2 + \|M^{(N)}\|_{\mathrm{F}}^2 - 2\|(L^{(N)})^{\intercal} M^{(N)} \|_*.
\]
By Lemma \ref{lem:frob-ssln}, the first two terms on the right hand side are of order $\frac{1}{2}d^2 N^2$. To prove \eqref{eqn:strict-gap}, we will show that the last term is also of order $N^2$.

For any square matrix $C$, we have the inequality
\[
\|C\|_{\mathrm{F}}^2 = \sum_i \sigma_i^2(C) \leq \sigma_1(C) \sum_i \sigma_i(C) = \|C\|_{2 \rightarrow 2} \|C\|_*.
\]
When $C$ is non-zero, we may divide both sides by $\|C\|_{2 \rightarrow 2}$ to get the bound $\|C\|_* \ge \|C\|_{\mathrm{F}}^2 / \|C\|_{2\to 2}$. Applying this to $L^{\intercal}M$ (where again $L = L^{(N)}$ and $M = M^{(N)}$) and using the submultiplicity of the operator norm, we have, almost surely,
\begin{equation} \label{eq:bures-bound}
\frac{1}{N^2} \|L^\intercal M\|_* \ge \frac{  \|L^\intercal M\|_{\mathrm{F}}^2 / N^3 }{ (\|L\|_{2\to 2} / \sqrt{N}) (\|M\|_{2\to 2} / \sqrt{N}) }.
\end{equation}

By Lemma \ref{lem:LTM.F.limit}, we have 
\begin{equation} \label{eqn:LTM.limit}
\lim_{N \rightarrow \infty} \frac{1}{N^3} \| L^{\intercal} M \|_{\mathrm{F}}^2 = \frac{d^3}{3} \quad \text{almost surely.}
\end{equation}

Using \eqref{eq:bures-bound}, \eqref{eqn:LTM.limit} and Lemma \ref{lem:opnorm.bound}, we have 
\[
\liminf_{N \rightarrow \infty} \frac{1}{N^2} \|L^{\intercal}M\|_* \geq \frac{d^3/3}{K^2} > 0.
\]
It follows that
\begin{align*}
\limsup_{N \rightarrow \infty} \frac{1}{N^2} \dist_{\mathrm{BW}}^2(A^{(N)}, B^{(N)}) &\leq \frac{1}{2}d^2 + \frac{1}{2}d^2 - 2\frac{d^3/3}{K^2} \\
&< d^2 \\
&= \lim_{N \rightarrow \infty} \frac{1}{N^2} \dist_{\mathrm{AW}}^2(L^{(N)}, M^{(N)})
\end{align*}
and the proof is complete.


\end{proof}

\section{Adapted Gelbrich bound} \label{sec:Gelbrich}
In classical optimal transport, the $2$-Wasserstein distance between Gaussian distributions serves as a lower bound of $\cW_2$ between arbitrary probability measures on $\bR^n$ with finite second moment. Specifically, we have the following result proved by Gelbrich \cite{G1990}.

\begin{theorem}[Theorem 2.1 in \cite{G1990}]
Let $\mu, \nu \in \cP_2(\mathbb{R}^n)$. Let $(a, A) \in \bR^n \times \mathscr{S}_+(n)$ be the mean and covariance of $\mu$ and similarly let $(b, B) \in \bR^n \times \mathscr{S}_+(n)$ be those of $\nu$. Then
\begin{equation} \label{eqn:Gelbrich.bound}
\begin{split}
\cW_2^2(\mu, \nu) &\geq \cW_2^2( \cN_{Nd}(a, A), \cN_{Nd}(b, B)) = \|a - b\|_2^2 + \dist_{\mathrm{BW}}^2(A, B).
\end{split}
\end{equation}
\end{theorem}

The lower bound \eqref{eqn:Gelbrich.bound}, which is analytically tractable, has found many applications including {\it distributionally robust optimization} \cite{kuhn2019wasserstein}.

It is natural to ask if the adapted analogue of \eqref{eqn:Gelbrich.bound} holds. That is, given $\mu, \nu \in \cP_2(\bR^{Nd})$ with mean-covariance pairs $(a, A)$, $(b, B)$, whether
\begin{equation} \label{eqn:adapted.Gelbrich.bound}
\begin{split}
\AW_2^2(\mu, \nu) &\geq \AW_2^2( \cN_{Nd}(a, A), \cN_{Nd}(b, B)) \\
&= \|a - b\|_2^2 + \dist_{\AW}^2( \cC_{\min}(A), \cC_{\min}(B)),
\end{split}
\end{equation}
where the last equality follows from Theorem \ref{thm:AW2.Gaussian.distributions}.

First, we show with an example that the adapted Gelbrich bound does not hold in general.

\begin{example} \label{eg:Gelbrich.counterexample}
Let $N = 2$ and $d = 1$. Let $\epsilon_1, \epsilon_2$ be independent $\cN_1(0, 1)$ random variables, and let $\delta > 0$ be arbitrary. Define 
\begin{equation} \label{eqn:eg.XY}
\begin{split}
X &= \begin{bmatrix} X_1 \\ X_2 \end{bmatrix} = \begin{bmatrix} \epsilon_1 \\ (\epsilon_1^2 - 1)/\sqrt{2} \end{bmatrix}, \\ 
Y &= \begin{bmatrix} Y_1 \\ Y_2 \end{bmatrix} = \begin{bmatrix}
(\epsilon_1^2 - 1)/\sqrt{2} \\
(2 + \delta) (\epsilon_1^2 - 1)/\sqrt{2} + \epsilon_2
\end{bmatrix}.
\end{split}
\end{equation}
Let $\mu = \cL(X)$ and $\nu = \cL(Y)$. It is easy to verify that the $\mu$ and $\nu$ have zero means, and have covariance matrices $A$ and $B$ in terms of the minimal Cholesky decomposition by
\[
A = \begin{bmatrix} 1 & 0 \\ 0 & 1 \end{bmatrix}, \quad B = \begin{bmatrix}1 & 0 \\ 2 + \delta & 1 \end{bmatrix}\begin{bmatrix}1 & 2 + \delta \\ 0 & 1 \end{bmatrix} = \begin{bmatrix} 1 & 2 + \delta \\ 2 + \delta & (2 + \delta)^2 + 1 \end{bmatrix}.
\]
By Theorem \ref{thm:AW2.Gaussian.distributions}, we have
\[
\AW_2^2(\cN_2(0, A), \cN_2(0, B)) = 2 + (2 + (2 + \delta)^2) - 2 (1 + 1) = (2 + \delta)^2.
\]

On the other hand, the coupling $\pi = \cL(X, Y)$ defined by \eqref{eqn:eg.XY} is an element of $\Cpl_{\mathrm{bc}}(\mu, \nu)$, under which
\[
\bE_{\pi}[\|X - Y\|_2^2] = 2 + (2 + (2 + \delta)^2) - 2(2 + \delta) = (2 + \delta)^2 - 2 \delta.
\]
It follows that $\AW_2(\mu, \nu) < \AW_2(\cN_2(0, A), \cN_2(0, B))$ and the adapted Gelbrich bound fails.
\end{example}

We identify a sufficient condition, which is rather restrictive, under which the adapted Gelbrich bound holds. 

\begin{assumption}[Martingale difference condition] \label{ass:martingale.difference.condition}
Let $\mu$ have mean $a \in \bR^{Nd}$ and covariance $A \in \mathscr{S}_{+}(Nd)$. Let $L = \cC_{\min}(A)$ and let $X \sim \mu$. Define
\begin{equation} \label{eqn:epsilon}
\epsilon := L^{\ominus} (X - a).
\end{equation}
We say that $\mu$ satisfies the martingale difference condition if $\epsilon$ is a martingale difference with respect to the filtration induced by $X$. That is,
\begin{equation} \label{eqn:martingale.difference.condition}
\bE[ \epsilon_t \mid X_{1:(t-1)}] = 0, \quad t \in [N].
\end{equation}
\end{assumption}

\begin{lemma} \label{lem:proj_cov}
Consider $\epsilon$ defined by \eqref{eqn:epsilon}. Then $\bE[\epsilon] = 0$ and
\[
\Cov(\epsilon) = \bE[\epsilon \epsilon^{\intercal}] \preceq I
\]
in Loewner order. 
\end{lemma}
\begin{proof}
Define $\cI = \cI(L)$ as in \eqref{eqn:cI}. Since $L^{\ominus}=E_{\cI} (L_{[\cI,\cI]})^{-1}E_{\cI}^\intercal$, we have
\[
\epsilon = E_{\cI} (L_{[\cI,\cI]})^{-1}(X_\cI-a_{\cI}),
\]
which clearly has zero mean. It follows that 
\begin{equation*}
\begin{split}
\Cov(\epsilon) &= E_{\cI} (L_{[\cI,\cI]})^{-1}\Cov(X_{\cI})((L_{[\cI,\cI]
})^{-1})^{\intercal}E_{\cI}^\intercal \\
&= E_{\cI} (L_{[\cI,\cI]})^{-1}A_{[\cI,\cI]}((L_{[\cI,\cI]})^{-1})^{\intercal}E_{\cI}^\intercal.
\end{split}
\end{equation*}
Since $A=LL^\intercal$ and $L_{[\dot, i]} = 0$ for all $i \notin \cI$, we have $A_{[\cI,\cI]}=L_{[\cI,\cI]}L_{[\cI,\cI]}^\intercal$. Hence $\Cov(\epsilon)=E_{\cI} E_{\cI}^\intercal \preceq I$. We remark that $E_{\cI} E_{\cI}^{\intercal}$ is the orthogonal projection onto the active coordinate subspace $\mathrm{span}\{e_i:i\in\mathcal{I}\}$.

\end{proof}

\begin{theorem}[Adapted Gelbrich bound under the martingale difference condition] \label{thm:adapted.Gelbrich}
Suppose that $\mu, \nu \in \cP_2(\bR^{Nd})$ satisfy Assumption \ref{ass:martingale.difference.condition}. Then the adapted Gelbrich bound \eqref{eqn:adapted.Gelbrich.bound} holds.
\end{theorem}
\begin{proof}
We assume without loss of generality that $\mu$ and $\nu$ have mean zero. Let $L, M \in \mathscr{L}_{+}(Nd)$ be respectively the minimal Cholesky factors of the covariance matrices $A$ and $B$ of $\mu$ and $\nu$. Let $\pi \in \Cpl_{\mathrm{bc}}(\mu, \nu)$ and $(X, Y) \sim \pi$. Define $\epsilon^X = L^{\ominus} X$ and $\epsilon^Y = M^{\ominus} Y$. By Lemma \ref{lem:proj_cov}, for each $t \in [N]$ we have

\begin{equation} \label{eqn:lem.consequence}
 \Cov(\epsilon_t^X) \preceq I \quad \text{and} \quad  \Cov(\epsilon_t^Y) \preceq I.  
\end{equation}

Consider the transport cost
\begin{equation} \label{eqn:adapted.Gelbrich.proof}
\begin{split}
\bE_{\pi}[ \|X - Y\|_2^2] &= \bE_{\pi}[ \| L\epsilon^X - M\epsilon^Y \|^2] \\
&= \|L\|_{\mathrm{F}}^2 + \|M\|_{\mathrm{F}}^2 - 2 \bE_{\pi} [ (\epsilon^X)^{\intercal} L^{\intercal} M \epsilon^Y].
\end{split}
\end{equation}
Write
\begin{equation*}
\begin{split}
\bE_{\pi} [ (\epsilon^X)^{\intercal} L^{\intercal} M \epsilon^Y] &= \sum_{s, t = 1}^N \bE[ (\epsilon_s^X)^{\intercal} (L^{\intercal} M)_{s,t} \epsilon_t^{Y} ] \\
  &= \sum_{s, t = 1}^N \tr ( (L^{\intercal} M)_{s,t} \Sigma_{s,t}^{\intercal} ),
\end{split}
\end{equation*}
where $\Sigma_{s,t} = \bE_{\pi}[ \epsilon_s^X (\epsilon_t^Y)^\intercal]$ is the covariance between $\epsilon_s^X$ and $\epsilon_t^Y$.

We claim that $\Sigma_{s, t} = 0$ for $s \neq t$. Suppose $s > t$ (the other case is similar). Since $\pi$ is bicausal, the conditional distribution of $X_s$ given $X_{1:(s-1)}$ and $Y_{1:t}$ is the same as that of $X_s$ given $X_{1:(s-1)}$. Using bicausality and the tower property, we have
\begin{equation*}
\begin{split}
\Sigma_{s,t} &= \bE_{\pi}[ \epsilon_s^X (\epsilon_t^Y)^{\intercal}] \\
&= \bE_{\pi} [\bE_{\pi} [ (L^\ominus X)_s (M^\ominus Y)_t^{\intercal} \mid X_{1:(s-1)}, Y_{1:t}]] \\
&= \bE_{\pi} [ \bE_{\pi} [ \epsilon_s^X \mid X_{1:(s-1)}] (\epsilon_t^Y)^{\intercal} ].
\end{split}
\end{equation*}
By the martingale difference property, the inner conditional expectation is zero. Hence, $\Sigma_{s,t} = 0$. Let $t \in [N]$. From \eqref{eqn:lem.consequence}, for any $u, v \in \bR^{d}$ with $\|u\|_2 = \|v\|_2 = 1$, we have, by the Cauchy--Schwarz inequality,
\begin{equation*}
\begin{split}
u^{\intercal} \Sigma_{t,t} v &= \bE_{\pi} [ (u^{\intercal}\epsilon_t^X) (v^{\intercal} \epsilon_t^Y)] \\
&\leq \bE_{\pi} [(u^{\intercal} \epsilon_t^X)^2]^{\frac{1}{2}} \bE_{\pi} [(v^{\intercal} \epsilon_t^Y)^2]^{\frac{1}{2}} \leq 1.
\end{split}
\end{equation*}
It follows that all singular values $\sigma_i(\Sigma_{t,t})$ of $\Sigma_{t,t}$ are bounded above by $1$.

By the trace inequality \eqref{eqn:trace.inequality}, we have
\begin{equation*}
\begin{split}
\tr ( (L^{\intercal} M)_{t,t} \Sigma_{t,t}^{\intercal} ) &\leq \sum_{i = 1}^d \sigma_i( (L^{\intercal}M)_{t,t}) \sigma_i(\Sigma_{t,t}) \\
&\leq \sum_{i = 1}^d \sigma_i((L^{\intercal} M)_{t,t}) = \| (L^{\intercal}M)_{t,t} \|_*.
\end{split}
\end{equation*}
Plugging this into \eqref{eqn:adapted.Gelbrich.proof}, we have
\begin{equation*}
\begin{split}
\bE_{\pi}[\|X - Y\|_2^2] 
  &\geq \|L\|_{\mathrm{F}}^2 + \|M\|_{\mathrm{F}}^2 - 2 \sum_{t = 1}^N \| (L^{\intercal} M)_{t,t} \|_* \\
  &= \AW_2^2( \cN_{Nd}(0, A), \cN_{Nd}(0, B)).
\end{split}
\end{equation*}
Taking the infimum over all $\pi\in\Cpl_{\mathrm{bc}}(\mu,\nu)$ yields the desired bound.
\end{proof}

\section*{Acknowledgment}
This research is partially supported by NSERC Discovery Grant RGPIN-2025-06021. We thank Beatrice Acciaio, Songyan Hou and Yifan Jiang for helpful discussions. Some results of this paper were presented in the workshop ``Probabilistic Mass Transport - from Schr\"{o}dinger to Stochastic Analysis'' held at the Erwin Schr\"{o}dinger International Institute for Mathematics and Physics. We thank the organizers and the participants for their comments. 

\appendix
\section{Additional proofs} \label{sec:proofs}

\begin{proof}[Proof of Proposition \ref{prop:trace}]
We will prove (i) and (ii) together. Let $\|P\|_{2 \rightarrow 2} \leq 1$. Then $\sigma_i(P) \leq 1$ for all $i$. By the trace inequality \eqref{eqn:trace.inequality}, we have
\[
\tr(CP) \leq \sum_{i = 1}^n \sigma_i(C) \sigma_i(P) \leq \sum_{i = 1}^n \sigma_i(C) = \|C\|_*.
\]

Using the given singular value decomposition $C = U \Sigma V^{\intercal}$, consider $\tilde{P} := V^{\intercal} P U$. We have
\[
\tr(CP) = \tr(U \Sigma V^{\intercal} V \tilde{P} U^{\intercal}) = \tr(\Sigma \tilde{P}) = \sum_{i = 1}^r \sigma_i(C) \tilde{P}_{[i, i]}.
\]
Since $U$ and $V$ are orthogonal, we have $\|\tilde{P}\|_{2 \rightarrow 2} \leq \|P\|_{2 \rightarrow 2} \leq 1$.  In particular, $|\tilde{P}_{[i, i]}| = |e_i^{\intercal} \tilde{P} e_i| \leq 1$. It follows that the maximum trace is $\|C\|_*$ (choose for example $P = VU^{\intercal}$, so that $\tilde{P} = I_n$), and $P \in \mathscr{P}(C)$ if and only if
\begin{equation}\label{eqn:diag-equality}
\tilde{P}_{[i,i]} = 1, \quad i = 1,\dots,r.
\end{equation}

It remains to show that \eqref{eqn:diag-equality} is equivalent to the condition in \eqref{eqn:trace.optimizer}. Write, in block form,
\[
\tilde{P} = \begin{bmatrix} \tilde{P}_{1,1} & \tilde{P}_{1,2} \\ \tilde{P}_{2,1} & \tilde{P}_{2,2} \end{bmatrix},
\]
where $\tilde{P}_{1,1} \in \bR^{r \times r}$. Since $\|\tilde{P}\|_{2 \rightarrow 2} \leq 1$, we also have $\|\tilde{P}_{1,1}\|_{2 \rightarrow 2} \leq 1$ and $\|\tilde{P}_{2,2}\|_{2 \rightarrow 2} \leq 1$.

Assume \eqref{eqn:diag-equality} holds. Since $\tr(\tilde{P}_{1,1}) = r$, the trace inequality \eqref{eqn:trace.inequality} (for the product $\tilde{P}_{1,1} \cdot I_r$) implies that all singular values of $\tilde{P}_{1,1}$ are $1$. It follows that $\tilde{P}_{1,1}$ is orthogonal. So, each column of $\tilde{P}_{1,1}$ is a unit vector. From \eqref{eqn:diag-equality}, we have that $\tilde{P}_{1,1} = I_r$. 

We claim further that the off-diagonal blocks $\tilde{P}_{1,2}$ and $\tilde{P}_{2,1}$ vanish. To see this, note that for any $x \in \bR^{r}$, we have, since $\|\tilde{P}\|_{2 \rightarrow 2} \leq 1$,
\[
\left\| \tilde{P} \begin{bmatrix} x \\ 0 \end{bmatrix} \right\|_2^2
= \left\| \begin{bmatrix} I_r x \\ \tilde{P}_{2,1} x \end{bmatrix} \right\|_2^2
= \|x\|_2^2 + \|\tilde{P}_{2,1}x\|_2^2
\le \left\|\begin{bmatrix}x\\0\end{bmatrix}\right\|_2^2
= \|x\|_2^2.
\]
It follows that $\tilde{P}_{2,1} = 0$. Similarly, we have $\tilde{P}_{1,2} = 0$. Thus $\tilde{P} = \diag(I_r, K)$ with $\|K\|_{2 \rightarrow 2} \leq 1$, and we have
\[
P = V \tilde{P} U^{\intercal} = \begin{bmatrix} V_1 & V_0\end{bmatrix} \begin{bmatrix} I_r & 0 \\ 0 & K \end{bmatrix} \begin{bmatrix} U_1 & U_0 \end{bmatrix}^{\intercal} = V_1 U_1^{\intercal} + V_0 K U_0^{\intercal},
\]
which is the condition in \eqref{eqn:trace.optimizer}. Conversely, it is immediate to check that any matrix $P$ of the above form attains the maximum value.
\end{proof}

\begin{proof}[Proof of Lemma \ref{lem:opnorm.bound}]
For notational simplicity, we suppress $N$ and write $L = L^{(N)}$ and $M = M^{(N)}$. In the following, we let $C > 0$ be a constant, independent from $N$, that may change from line to line.

Consider the bound for $\|L\|_{2 \rightarrow 2}$; the case for $\|M\|_{2 \rightarrow 2}$ is the same. The main idea is to compare the block lower triangular matrix $L$ with a full random matrix. For each $N$, we may express $L$ in the form
\[
L_{[i,j]} = a_{[i,j]} G_{[i,j]}, \quad i, j \in [Nd],
\]
where
\begin{equation} \label{eqn:aij}
a_{[i,j]} = \left\{\begin{array}{ll}
        1, & \text{if the entry $L_{[i,j]}$ is in a block $L_{s,t}$ with $s \geq t$}\\
        0, & \text{otherwise}
        \end{array}\right.
\end{equation}
and $G_{[i,j]}$ are i.i.d.~$\cN(0, 1)$ random variables. Here, we assume without loss of generality that additional i.i.d.~Gaussian random variables (corresponding to the upper triangular part) are defined on the given probability space.

By a theorem of Lata{\l}a \cite[Theorem 1]{latala2005some},\footnote{This result assumes that $G_{ij}$ are i.i.d.~$\cN(0, 1)$. Nevertheless, we note that a similar estimate \cite[Theorem 2]{latala2005some} holds for as long as the entries are independent and have finite fourth moment.} we have
\begin{equation} \label{eqn:Latala}
\bE [\|L\|_{2\to2}]
\leq
C\left(
\max_i  \Big( \sum_j a_{[i,j]}^2 \Big)^{1/2}
+
\max_j  \Big( \sum_i a_{[i,j]}^2 \Big)^{1/2}
+
\Big(\sum_{i,j}a_{[i,j]}^4\Big)^{1/4}
\right).
\end{equation}
From \eqref{eqn:aij}, we see that the $\max_i$ and $\max_j$ terms are both $Nd$, and that 
\begin{equation} \label{eqn:EL.upper.bound}
\sum_{i, j \in [Nd]} a_{ij}^4 = d^2 \binom{N+1}{2}.
\end{equation}
It follows from \eqref{eqn:Latala} that 
\[
\bE[\|L\|_{2 \rightarrow 2}] \leq C \sqrt{Nd}.
\]

Next, we show concentration of $\|L\|_{2 \rightarrow 2}$ about its mean. Observe that
\[
\|L\|_{2 \rightarrow 2} = \max_{u, v \in \bR^{Nd}: \|u\|_2 = \|v\|_2 = 1} \sum_{i, j \in [Nd]} a_{[i,j]} u_{[i]} v_{[j]} G_{[i,j]}
\]
is the maximum of the centered Gaussian process $(u^{\intercal} L v)_{u, v}$ indexed by unit vectors $u$ and $v$. Since
\begin{align}
\mathrm{Var} \Big( \sum_{i, j} a_{[i,j]} u_{[i]} v_{[j]} G_{[i,j]} \Big) &= \sum_{i, j} a_{[i,j]}^2 u_{[i]}^2 v_{[j]}^2 \leq \sum_i u_{[i]}^2 \sum_j v_{[j]}^2 = 1
\end{align}
for all unit vectors $u, v$, 
the celebrated Borell--TIS inequality (see e.g.~\cite[Theorem 2.1.1]{AT07}) implies that
\[
\bP\left( \|L\|_{2 \rightarrow 2} \geq \bE[\|L\|_{2 \rightarrow 2}] + z \right) \leq e^{-z^2/2}, \quad z > 0.
\]
Choosing $z = 2 \sqrt{\log N}$, we obtain
\[
\bP\left( \|L\|_{2 \rightarrow 2} \geq \bE[\|L\|_{2 \rightarrow 2}] + 2 \sqrt{\log N} \right) \leq \frac{1}{N^2}, \quad N \geq 1.
\]
Since $\sum_{N \geq 1} N^{-2} < \infty$, by Borel--Cantelli we have, almost surely,
\[
\|L\|_{2 \rightarrow 2} \leq \bE[\|L\|_{2 \rightarrow 2}] + 2 \sqrt{\log N}, \quad N \text{ sufficiently large}.
\]
Combining this with \eqref{eqn:EL.upper.bound} (and using $\sqrt{\log N} \leq \sqrt{N}$), we have almost surely that
\[
\|L\|_{2 \rightarrow 2} \leq C \sqrt{N}, \quad N \text{ sufficiently large}.
\]

\end{proof}

\begin{proof}[Proof of Lemma \ref{lem:LTM.F.limit}]
We write, as before, $L=L^{(N)}$ and $M=M^{(N)}$. We first compute the conditional expectation given $L$. Since
\[
\|L^{\intercal}M\|_{\mathrm{F}}^2
=
\tr(L^{\intercal}MM^{\intercal}L),
\]
we have, by independence of $L$ and $M$,
\[
\bE\big[\|L^{\intercal}M\|_{\mathrm{F}}^2\mid L\big]
=
\tr\big(L^{\intercal}\bE[MM^{\intercal}]L\big).
\]
Now $\bE[MM^{\intercal}]$ is block diagonal. Indeed, the $s$-th row-block of $M$ contains exactly $s$ non-zero $d\times d$ Gaussian blocks, and for each such block $G$ we have $\bE[GG^{\intercal}]=dI_d$. Hence the $s$-th diagonal block of $\bE[MM^{\intercal}]$ is equal to $sd\,I_d$. It follows that
\[
\bE\big[\|L^{\intercal}M\|_{\mathrm{F}}^2\mid L\big]
=
d\sum_{s=1}^N s\sum_{t=1}^s \|L_{s,t}\|_{\mathrm{F}}^2.
\]

For each $s\in[N]$, the random variable $\sum_{t=1}^s \|L_{s,t}\|_{\mathrm{F}}^2$ is chi-square with $sd^2$ degrees of freedom. In particular,
\[
\bE\left[\sum_{t=1}^s \|L_{s,t}\|_{\mathrm{F}}^2\right]=sd^2,
\qquad
\Var\left(\sum_{t=1}^s \|L_{s,t}\|_{\mathrm{F}}^2\right)=2sd^2.
\]
Moreover, these random variables are independent as $s$ varies. Therefore
\[
\sum_{s=1}^{\infty}
\frac{
\Var\left(\sum_{t=1}^s \|L_{s,t}\|_{\mathrm{F}}^2-sd^2\right)
}{s^4}
=
2d^2\sum_{s=1}^{\infty}\frac{1}{s^3}
<\infty.
\]
By Kolmogorov's convergence criterion, the series
\[
\sum_{s=1}^{\infty}
\frac{
\sum_{t=1}^s \|L_{s,t}\|_{\mathrm{F}}^2-sd^2
}{s^2}
\]
converges almost surely. An application of Kronecker's lemma yields
\[
\frac{1}{N^3}\sum_{s=1}^N
s\left(\sum_{t=1}^s \|L_{s,t}\|_{\mathrm{F}}^2-sd^2\right)
\to 0
\qquad\text{almost surely.}
\]
Consequently,
\[
\frac{1}{N^3}\bE\big[\|L^{\intercal}M\|_{\mathrm{F}}^2\mid L\big]
=
\frac{d^3}{N^3}\sum_{s=1}^N s^2
+
\frac{d}{N^3}\sum_{s=1}^N
s\left(\sum_{t=1}^s \|L_{s,t}\|_{\mathrm{F}}^2-sd^2\right),
\]
and therefore
\[
\frac{1}{N^3}\bE\big[\|L^{\intercal}M\|_{\mathrm{F}}^2\mid L\big]
\to
\frac{d^3}{3}
\qquad\text{almost surely.}
\]

It remains to show that the fluctuation around the conditional expectation is negligible on the scale $N^3$. Conditionally on $L$, the matrix $M$ is Gaussian and the map $H\mapsto L^{\intercal}H$ is linear. Hence there exist a standard Gaussian vector $g$ and a positive semidefinite matrix $A$ such that
\[
\|L^{\intercal}M\|_{\mathrm{F}}^2=g^{\intercal}Ag.
\]
If $\lambda_i$ are the eigenvalues of $A$, then
\[
\Var\big(\|L^{\intercal}M\|_{\mathrm{F}}^2\mid L\big)
=
2\sum_i \lambda_i^2
\le
2\Big(\max_i\lambda_i\Big)\sum_i\lambda_i.
\]
Now
\[
\sum_i\lambda_i=\tr(A)=\bE\big[\|L^{\intercal}M\|_{\mathrm{F}}^2\mid L\big],
\]
while
\[
\max_i\lambda_i=\|A\|_{2\to 2}\le \|L\|_{2\to 2}^2,
\]
because the linear map $H\mapsto L^{\intercal}H$ has operator norm at most $\|L\|_{2\to 2}$ with respect to the Frobenius norm. Thus
\[
\Var\big(\|L^{\intercal}M\|_{\mathrm{F}}^2\mid L\big)
\le
2\|L\|_{2\to 2}^2\,\bE\big[\|L^{\intercal}M\|_{\mathrm{F}}^2\mid L\big].
\]

By Lemma \ref{lem:opnorm.bound}, there exists a constant $K>0$ and an event of probability one such that
\[
\|L\|_{2\to 2}^2\le K^2N
\]
for all sufficiently large $N$. By the first part of the proof, there exists an event of probability one such that
\[
\frac{1}{N^3}\bE\big[\|L^{\intercal}M\|_{\mathrm{F}}^2\mid L\big]
\le \frac{d^3}{3}+1
\]
for all sufficiently large $N$. Intersecting these two events, we obtain an event of probability one on which
\[
\Var\big(\|L^{\intercal}M\|_{\mathrm{F}}^2\mid L\big)
\le
2K^2\left(\frac{d^3}{3}+1\right)N^4
\]
for all sufficiently large $N$.

 Chebyshev's inequality now gives (after conditioning), for every $\varepsilon>0$,
\[
\bP\left(
\left.
\left|
\|L^{\intercal}M\|_{\mathrm{F}}^2
-
\bE\big[\|L^{\intercal}M\|_{\mathrm{F}}^2\mid L\big]
\right|
>
\varepsilon N^3
\,\right|\,L
\right)
\le
\frac{2K^2\left(\frac{d^3}{3}+1\right)}{\varepsilon^2 N^2}
\]
for all sufficiently large $N$, almost surely. Since $\sum_{N\ge1}N^{-2}<\infty$, given $L$, the Borel--Cantelli lemma yields
\[
\frac{
\|L^{\intercal}M\|_{\mathrm{F}}^2
-
\bE\big[\|L^{\intercal}M\|_{\mathrm{F}}^2\mid L\big]
}{N^3}
\to 0
\qquad\text{almost surely.}
\]

Combining this with the limit for the conditional expectation, we conclude that
\[
\frac{1}{N^3}\|L^{\intercal}M\|_{\mathrm{F}}^2
\to
\frac{d^3}{3}
\qquad\text{almost surely.}
\]
\end{proof}

\bibliographystyle{plain}
\bibliography{transport.bib}
\end{document}